\documentclass[smallextended]{svjour3}
\smartqed

\usepackage{amsfonts, amsmath, empheq, amsthm}
\usepackage[cal=rsfs]{mathalpha}
\usepackage{graphicx}
\usepackage{color}
\usepackage{caption}
\usepackage{subfig}
\usepackage{bm}
\usepackage{verbatim}
\usepackage{hyperref}
\usepackage{bmpsize}
\usepackage{float} 
\usepackage{cases}
\restylefloat{table}
\usepackage[table, dvipsnames]{xcolor}
\usepackage[margin=1.0in]{geometry}
\usepackage[belowskip=-15pt,aboveskip=0pt]{caption}
\usepackage{makecell}
\usepackage{stackengine,scalerel}
\usepackage{verbatim}
\usepackage[mathscr]{euscript}
\usepackage{cite}

\newcommand\overstar[1]{\ThisStyle{\ensurestackMath{%
			\setbox0=\hbox{$\SavedStyle#1$}%
			\stackengine{0pt}{\copy0}{\kern.2\ht0\smash{\SavedStyle*}}{O}{c}{F}{T}{S}}}}

\allowdisplaybreaks

\newtheorem{algorithm}[theorem]{Algorithm}

\newcommand{\be}{\begin{equation}}
\newcommand{\ee}{\end{equation}}
\newcommand{\bea}{\begin{eqnarray}}
\newcommand{\eea}{\end{eqnarray}}
\newcommand{\beas}{\begin{eqnarray*}}
\newcommand{\eeas}{\end{eqnarray*}}

\newcommand{\vertiii}[1]{{\left\vert\kern-0.25ex\left\vert\kern-0.25ex\left\vert #1 
    \right\vert\kern-0.25ex\right\vert\kern-0.25ex\right\vert}}

\newcommand{\ub}{\boldsymbol{u}_j}
\newcommand{\pb}{{p}_j}
\newcommand{\Bb}{\boldsymbol{B}_j}
\newcommand{\lamb}{{\lambda}_j}
\newcommand{\fb}{\boldsymbol{f}_j}
\newcommand{\curlGb}{\nabla \times \boldsymbol{g}_j}
\newcommand{\nuPert}{{\nu}_j}
\newcommand{\gamPert}{{\gamma}_j}

\newcommand{\fn}[1]{\boldsymbol{f}^{#1}_{j,h}}
\newcommand{\fe}[1]{\boldsymbol{f}^{#1}_j}
\newcommand{\fbe}[1]{{\boldsymbol{f}^{#1}_b}_j}

\newcommand{\gn}[1]{\nabla\times\boldsymbol{g}^{#1}_{j,h}}
\newcommand{\geNew}[1]{\nabla\times\boldsymbol{g}^{#1}_j}
\newcommand{\gbe}[1]{\nabla\times{\boldsymbol{g}^{#1}_b}_j}

\newcommand{\un}[1]{\boldsymbol{u}^{#1}_{j,h}}
\newcommand{\ue}[1]{\boldsymbol{u}^{#1}_j}
\newcommand{\uhe}[1]{\hat{\boldsymbol{u}}^{#1}_j}
\newcommand{\uue}[1]{\breve{\boldsymbol{u}}^{#1}_j}

\newcommand{\bn}[1]{\boldsymbol{B}^{#1}_{j,h}}
\newcommand{\Be}[1]{\boldsymbol{B}^{#1}_j}
\newcommand{\bhe}[1]{\hat{\boldsymbol{B}}^{#1}_j}
\newcommand{\bue}[1]{\breve{\boldsymbol{B}}^{#1}_j}

\newcommand{\pn}[1]{p^{#1}_{j,h}}
\newcommand{\pe}[1]{p^{#1}_j}
\newcommand{\phe}[1]{\hat{p}^{#1}_j}
\newcommand{\pue}[1]{\breve{p}^{#1}_j}

\newcommand{\lamn}[1]{\lambda^{#1}_{j,h}}
\newcommand{\lame}[1]{\lambda^{#1}_j}
\newcommand{\lamhe}[1]{\hat{\lambda}^{#1}_j}
\newcommand{\lamue}[1]{\breve{\lambda}^{#1}_j}

\newcommand{\bs}[3]{b^{*}(#1,#2,#3)}

\newcommand{\norm}[1]{\lVert#1\rVert}
\newcommand{\normiii}[1]{{\left\vert\kern-0.25ex\left\vert\kern-0.25ex\left\vert #1 
		\right\vert\kern-0.25ex\right\vert\kern-0.25ex\right\vert}}
		
\newcommand{\vt}{\tilde{\boldsymbol{v}}^{n+1/2}}
\newcommand{\vh}[1]{\boldsymbol{v}^{#1}}
\newcommand{\vhbar}[1]{\bar{\boldsymbol{v}}^{#1}_j}

\newcommand{\vhe}[1]{\hat{\boldsymbol{v}}^{#1}_j}
\newcommand{\vue}[1]{\breve{\boldsymbol{v}}^{#1}_j}
\newcommand{\vhbdf}[1]{\overstar{\boldsymbol{v}}^{#1}}
\newcommand{\vtt}{{\tilde{\boldsymbol{v}}}^{n+1}}
\newcommand{\vhbarbdf}[1]{\bar{\boldsymbol{v}}^{#1}_j}

\newcommand{\ut}{\tilde{\boldsymbol{u}}^{n+1/2}_{j,h}}
\newcommand{\bt}{\tilde{\boldsymbol{B}}^{n+1/2}_{j,h}}
\newcommand{\ute}{\tilde{\boldsymbol{u}}^{n+1/2}_{j}}
\newcommand{\bte}{\tilde{\boldsymbol{B}}^{n+1/2}_{j}}
\newcommand{\qun}[1]{q^{#1}_{j,h}}
\newcommand{\qu}[1]{q_j^{#1}}
\newcommand{\ru}[1]{R^{#1}_{j}}
\newcommand{\run}[1]{R^{#1}_{j,h}}

\newcommand{\bbar}[1]{\bar{\boldsymbol{B}}^{#1}_{j}}
\newcommand{\bbn}[1]{\bar{\boldsymbol{B}}^{#1}_{j,h}}
\newcommand{\ubar}[1]{\bar{\boldsymbol{u}}^{#1}_{j}}
\newcommand{\ubn}[1]{\bar{\boldsymbol{u}}^{#1}_{j,h}}
\newcommand{\rubdf}[1]{\overstar{R}^{#1}_{j}}
\newcommand{\runbdf}[1]{\overstar{R}^{#1}_{j,h}}

\newcommand{\utt}{{\tilde{\boldsymbol{u}}}^{n+1}_{j,h}}
\newcommand{\btt}{{\tilde{\boldsymbol{B}}}^{n+1}_{j,h}}
\newcommand{\utte}{{\tilde{\boldsymbol{u}}}^{n+1}_{j}}
\newcommand{\btte}{{\tilde{\boldsymbol{B}}}^{n+1}_{j}}

\newcommand{\nuM}{\bar{\nu}^n}
\newcommand{\nuP}{{\nu^{\prime n}_j}}
\newcommand{\nuPMax}{{\nu^{\prime}_{\max}}}

\newcommand{\gamm}{{\bar{\gamma}}^n}
\newcommand{\gamP}{{{\gamma}^{\prime n}_j}}
\newcommand{\gamPMax}{{{\gamma}^{\prime}_{\max}}}

\newcommand{\br}{\nonumber\\}

\begin{document}
\title{Second order, unconditionally stable, linear ensemble algorithms for the magnetohydrodynamics equations}

\titlerunning{Ensemble methods for MHD equations }

\author{
John Carter, Daozhi Han and Nan Jiang
}

\institute{J. Carter \at
	Department of Mathematics and Statistics, Missouri University of Science and Technology, Rolla, MO 65409, USA.    
	\email{jachdm@mst.edu}             \\
	D. Han \at
	Department of Mathematics and Statistics, Missouri University of Science and Technology, Rolla, MO 65409, USA.    
	\email{handaoz@mst.edu}             \\
	N. Jiang \at
	Department of Mathematics, University of Florida, Gainesville, FL.\\
	\email{ jiangn@ufl.edu} 
}

\maketitle

\begin{abstract}
We propose two unconditionally stable, linear ensemble algorithms with pre-computable shared coefficient matrices across different realizations for the magnetohydrodynamics equations.  The viscous terms are treated by a standard perturbative discretization.  The nonlinear terms are discretized fully explicitly within the framework of the generalized positive auxiliary variable approach (GPAV). Artificial viscosity stabilization that modifies the kinetic energy is introduced to  improve  accuracy of the GPAV ensemble methods.  Numerical results are presented to demonstrate the accuracy and robustness of the ensemble algorithms.

\end{abstract}

\section{Introduction}
Magnetohydrodynamics (MHD) flow  describes electrically conducting fluid moving through a magnetic field. It has important applications in fusion technology, submarine propulsion system, liquid metals in magnetic pumps, and so on. The mathematical model comprises the Navier-Stokes equations for fluid flow and Maxwell's equations for electromagnetics.  In practical applications, the problem parameters such as viscosity and magnetic resistivity, external body forcing and initial conditions, are invariably subject to uncertainty. To quantify the impact of uncertainty and develop high-fidelity numerical simulations, one usually computes the flow ensembles in which the MHD equations are solved repeatedly with different inputs.  The aim of this article is to develop efficient second-order accurate ensemble algorithms that are unconditionally stable and suitable for long-time simulations.
Therefore we consider solving $J$ times the following MHD equations: 
 for $j=1,2,...,J$,
\begin{equation}\label{eq:MHD}
	\left\{\begin{aligned}
	&\boldsymbol{u}_{j,t}+\boldsymbol{u}_j\cdot\nabla \boldsymbol{u}_j  -s\boldsymbol{B}_j\cdot \nabla \boldsymbol{B}_j - \nu_j \Delta \boldsymbol{u}_j + \nabla p_j  
	=\boldsymbol{f}_j \text{ in } \Omega\times (0,T),\\
	&
	\nabla \cdot \boldsymbol{u}_j = 0, \text{ in } \Omega\times (0,T),\\
	&
	\boldsymbol{B}_{j,t} + \boldsymbol{u}_j\cdot \nabla \boldsymbol{B}_j - \boldsymbol{B}_j\cdot \nabla \boldsymbol{u}_j - \gamma_j \Delta \boldsymbol{B}_j + \nabla \lambda_j
	=
	\nabla \times \boldsymbol{g}_j \text{ in } \Omega\times (0,T),\\
	&
	\nabla \cdot \boldsymbol{B}_j = 0, \text{ in } \Omega\times (0,T),\\
	&
	\boldsymbol{u}_j(x,0)   =\boldsymbol{u}_j^{0}(x),   \text{ in } \Omega,\quad \boldsymbol{B}_j(x,0)    = \boldsymbol{B}_j^{0}(x), \text{ in } \Omega.
	\end{aligned}\right.
\end{equation}
Here $\boldsymbol{u}_j$ is the fluid velocity, $p_j$ the pressure, $\boldsymbol{B}_j$ the magnetic field and $\lambda_j$ is a Lagrange multiplier corresponding to the solenoidal constraint on $\boldsymbol{B}_j$ \cite{MoRe2017}. The body force $\boldsymbol{f}_j(x,t)$  and $\nabla \times \boldsymbol{g}_j$ are given,  $s$ is the coupling number, $\nu_j$ is the kinematic viscosity, and $\gamma_j$ is the magnetic resistivity.  Dirichlet boundary conditions will be imposed for both $\boldsymbol{u}_j$ and $\boldsymbol{B}_j$, though the numerical methods are also applicable  to other boundary conditions including  $\nabla \times \boldsymbol{B}_j=0$ on $\partial \Omega$. Note that we have adopted an equivalent  formulation  of the MHD equations, cf. \cite{WeZh2017, Trenchea2014, MoRe2017, MWRM2021}.

Ensemble methods have been extensively developed for solving the Navier-Stokes equations and related fluid models \cite{JiLa2014, JiLa2015, Jiang2015, Jiang2017, GJS2017,  Fiordilino2018, Jiang2019, GJW2019,  JiQi2019, JLY2021}.  The central idea in these ensemble methods is a perturbative time discretization that utilizes the ensemble mean corrected by explicit treatment of the fluctuations in time marching of each realization. 
As  a result, at each time step the coefficient matrix  of the resulting linear systems is identical for all realizations, saving both storage and computational cost.  Moreover, under some constraint on the time-step and the size of fluctuations it is shown that the ensemble algorithms are long-time stable.  A similar ensemble method is developed in \cite{JiSc2018} and \cite{CaNa2021} for solving a reduced MHD system at low magnetic Reynolds number.
Based on the Elsasser formulation \cite{Elsasser1950} and the perturbative time discretization, a first-order decoupled and unconditionally stable ensemble algorithm is proposed and analyzed in \cite{MoRe2017, MWRM2021} for solving the full MHD model.  
An artificial eddy viscosity term is employed to ensure unconditional stability.  Due to the usage of Elsasser variables, the method appears to be limited to the case of Dirichlet boundary conditions.

Further computational efficiency gains can be achieved by fully explicit discretization of the nonlinear terms so that the exact same coefficient matrix is shared across different time steps in ensemble simulations.  This approach would often incur a CFL condition that hinders the efficiency of the algorithm for long-time simulation or for problems involving multiple scales.  One remedy  is the introduction of a Lagrange multiplier for enforcement of the underlying energy estimate (energy dissipation or conservation). This idea leads to recent development of the so-called   Invariant Energy Quadratization (IEQ) method  \cite{GuTi2013, YZW2017, YaJu2017, GZW2020}, and the Scalar Auxiliary Variable (SAV) approach \cite{SXY2018, SXY2019} for solving phase field models.  Extensions of these methods are  reported in \cite{YaDo2019, YaDo2020, Yang2021, LSZ2021}  on the design of linear, decoupled, unconditionally stable numerical schemes for solving general nonlinear equations satisfying an energy law.  Based on the SAV approach proposed in \cite{YaDo2019}, a stabilized SAV ensemble algorithm is developed in \cite{JiYa2021} for parameterized flow problems where superior accuracy is observed thanks to a penalization of the kinetic energy causing the high frequency mode to quickly roll-off in the energy spectrum \cite{ LLMNR2009}.  Stability and error analysis of a SAV method for the MHD equations is recently conducted in \cite{LWS2022}.

In this article we propose two linear, second-order accurate, unconditionally stable ensemble methods with shared coefficient matrix across different realizations and time steps for solving the MHD model. The parameters are treated by the usual perturbative method. We employ the Generalized Positive Auxiliary Variable framework (GPAV) from \cite{YaDo2020} in the discretization of the nonlinear terms. The advantages of the GPAV method include: linearity of the algebra equation for the scalar variable; provable positivity of the scalar variable; and flexibility in handling complex boundary conditions.  These Lagrange multiplier type approaches often suffer from poor accuracy especially for long time simulation of advection-dominated flow, cf. \cite{ZOWW2020} for a careful benchmark comparison study of the SAV approach. This drop in accuracy is also discussed and demonstrated in the numerical tests from \cite{YaDo2020}. In \cite{JZZ2022} a post-processing procedure is introduced to improve accuracy of the SAV method for the Cahn-Hilliard equation. In our method we adopt the stabilization technique of artificial viscosity that proves robust and efficient in past studies \cite{LLMNR2009, JiYa2021}.  The stabilization introduces a penalty term in the kinetic energy which leads to a quick roll-off of the under-resolved modes in the energy spectrum thus curtailing the inertial range and making the system more computable, cf. \cite{LLMNR2009}. This mechanism is well-known in the Navier-Stokes-$\alpha$ model for large eddy simulation of turbulence \cite{FHT2001, CHMZ1999}. We perform extensive numerical tests to gauge the accuracy, efficiency and robustness of the proposed ensemble methods.

To start, we define the ensemble mean and the fluctuation of the viscosity terms ${\nu}_j^n$ and the electric potential ${\gamma}_j^n$ at timestep $n$ respectively
\begin{align}
\nuM = \frac{1}{J}\sum_{j=1}^J \nu_j^n \quad&\text{ and } \quad \gamm = \frac{1}{J}\sum_{j=1}^J {\gamma}_j^n, \tag{mean}\\
 \nuP = \nu_j^n-\bar{\nu}^n  \quad&\text{ and } \quad \gamP = {\gamma}_j^n-{\bar{\gamma}}^n, \tag{fluctuation}\\
 \nuPMax = \max_j\max_{x\in\Omega}|\nuP(x)| \quad&\text{ and } \quad \gamPMax = \max_j\max_{x\in\Omega}|\gamP(x)|,\nonumber
\end{align}
where in our considerations $\nu_j^n= \nu_j$, $\gamma_j^n=\gamma_j$ are constants and $t_n= n\Delta t$ ($n=0,1,2, ...$).
Define
\begin{align}
	\vh{n+1/2} &= \frac{1}{2}(\vh{n+1}+\vh{n}),\qquad \vt = 2\vh{n-1/2}-\vh{n-3/2}, \\
	\qquad \vhbdf{n+1/2} &= \frac{3}{2}\vh{n} - \frac{1}{2}\vh{n-1},\qquad \vtt = 2\vh{n}-\vh{n-1}.
\end{align}

We define a shifted energy of the form
\begin{align}\label{shiftEner}
	E_j(t) = E[\boldsymbol{u}_j, \boldsymbol{B}_j] = \int_{\Omega} \frac{1}{2}|\boldsymbol{u}_j|^2 d\Omega + \int_{\Omega} \frac{s}{2}|\boldsymbol{B}_j|^2 d\Omega + C_0,
\end{align}
where $E[\boldsymbol{u}_j, \boldsymbol{B}_j]$ is the total kinetic energy of the system, which for physical examples is bounded from below, and $C_0$ is an arbitrarily small positive constant chosen in such a way that $E_j(t) > 0$ for $0\leq t \leq T$. Next, let $\mathcal{F}$ be any one-to-one increasing differentiable function with $\mathcal{F}^{-1} = \mathcal{G}$ such that
\begin{numcases}{}
	\mathcal{F}(\chi)>0,\quad  \chi>0, \label{incFuncF}\\
	\mathcal{G}(\chi)>0,\quad  \chi>0. \label{incFuncG}
\end{numcases}
The scalar variable $R_j(t)$ is defined by
\begin{align}
	R_j(t) &= \mathcal{G}(E_j),\\
	E_j(t) &= \mathcal{F}(R_j).
\end{align}
With $E_j$ as in \eqref{shiftEner}, $R_j(t)$ then satisfies
\begin{align}\label{FEnergyEq}
	\mathcal{F}^{\prime}(R_j)\frac{dR_j}{dt} = \int_{\Omega} \boldsymbol{u}_j\cdot \frac{\partial \ub}{\partial t} d\Omega + \int_{\Omega} s\boldsymbol{B}_j\cdot \frac{\partial \Bb}{\partial t} d\Omega.
\end{align}
Since $\frac{\mathcal{F}(R_j)}{E_j} = 1$ for all $j$, we may write
\begin{align}
	\mathcal{F}^{\prime}(R_j)\frac{d R_j}{dt} &= \int_{\Omega} \left[ \boldsymbol{u}_j\cdot \frac{\partial \ub}{\partial t} + s\boldsymbol{B}_j\cdot \frac{\partial \Bb}{\partial t}\right] d\Omega + \left[\frac{\mathcal{F}(R_j)}{E_j} - 1 \right] \bigg[ \int_{\Omega} \ub \cdot\bigg( \nuPert \Delta \ub - \nabla \pb + \fb \bigg) d\Omega \\
	&\quad + \int_{\Omega} s\Bb \cdot\bigg( \gamPert \Delta \Bb - \nabla \lamb + \curlGb \bigg) d\Omega \bigg] \br
	&\quad + \frac{\mathcal{F}(R_j)}{E_j} \bigg[ \int_{\Omega} \ub \cdot[\Bb \cdot \nabla \Bb - \ub \cdot \nabla \ub] d\Omega - \int_{\Omega} \ub \cdot[\Bb \cdot \nabla \Bb - \ub \cdot \nabla \ub] d\Omega \br
	&\quad + \int_{\Omega} s\Bb \cdot[\Bb \cdot \nabla \ub - \ub \cdot \nabla \Bb] d\Omega - \int_{\Omega} s\Bb \cdot[\Bb \cdot \nabla \ub - \ub \cdot \nabla \Bb] d\Omega \bigg] \br
	&= \int_{\Omega} \left[ \boldsymbol{u}_j\cdot \frac{\partial \ub}{\partial t} + s\boldsymbol{B}_j\cdot \frac{\partial \Bb}{\partial t}\right] d\Omega \br
	&\quad- 
	\int_{\Omega} \ub \cdot\bigg( \nuPert \Delta \ub - \nabla \pb + \frac{\mathcal{F}(R_j)}{E_j}[\Bb \cdot \nabla \Bb - \ub \cdot \nabla \ub]
	+ \fb \bigg)d\Omega \br
	&\quad- 
	\int_{\Omega} s\Bb \cdot\bigg( \gamPert \Delta \Bb - \nabla \lamb + \frac{\mathcal{F}(R_j)}{E_j}[\Bb \cdot \nabla \ub - \ub \cdot \nabla \Bb]
	+ \curlGb \bigg)d\Omega \br
	&\quad 
	+ \frac{\mathcal{F}(R_j)}{E_j}\bigg[ \int_{\Omega} \ub \cdot[\Bb \cdot \nabla \Bb - \ub \cdot \nabla \ub + \fb] d\Omega \br
	&\quad+ \int_{\Omega} s\Bb \cdot[\Bb \cdot \nabla \ub - \ub \cdot \nabla \Bb + \curlGb] d\Omega \bigg]. \nonumber
\end{align}
Note that all the additional terms above amount to adding zero to \eqref{FEnergyEq}. Using integration by parts we get the equality
\begin{align}
	&\int_{\Omega} \ub \cdot[\Bb \cdot \nabla \Bb - \ub \cdot \nabla \ub + \fb] d\Omega + \int_{\Omega} s\Bb \cdot[\Bb \cdot \nabla \ub - \ub \cdot \nabla \Bb + \curlGb] d\Omega \\
	&\qquad = -\int_{\Omega} (\nuPert |\nabla \ub|^2 + s \gamPert |\nabla \Bb|^2 ) d\Omega + \int_{\Omega} (\fb \cdot \ub + s (\curlGb) \cdot \Bb ) d\Omega + \int_{\Gamma} B_{S}( \ub, \Bb ) d\Gamma, \nonumber
\end{align}
where $B_{S}( \ub, \Bb )$ represents the forcing terms on the boundary, defined as
\begin{align}
	B_{S}( \ub, \Bb ) &= \int_{\Gamma} \bigg( -\frac{1}{2}|\ub|^2 \ub - \frac{s}{2}|\Bb|^2 \ub + \nuPert \nabla \ub\cdot \ub - \pb \ub  \\
	&\quad +  s(\Bb \cdot \ub)\Bb + s\gamPert \nabla\Bb \cdot \Bb - s\lamb \Bb \bigg)\cdot \hat{n} \hspace{0.5em} d\Gamma \nonumber
\end{align}
and $\hat{n}$ is the unit normal vector to the boundary. We use this equality and write
\begin{align}
	\mathcal{F}^{\prime}(R_j)\frac{d R_j}{dt} &= \int_{\Omega} \left[ \boldsymbol{u}_j\cdot \frac{\partial \ub}{\partial t} + s\boldsymbol{B}_j\cdot \frac{\partial \Bb}{\partial t}\right] d\Omega \\
	&\quad- 
	\int_{\Omega} \ub \cdot\bigg( \nuPert \Delta \ub - \nabla \pb + \frac{\mathcal{F}(R_j)}{E_j}[\Bb \cdot \nabla \Bb - \ub \cdot \nabla \ub]
	+ \fb \bigg)d\Omega \br
	&\quad- 
	\int_{\Omega} s\Bb \cdot\bigg( \gamPert \Delta \Bb - \nabla \lamb + \frac{\mathcal{F}(R_j)}{E_j}[\Bb \cdot \nabla \ub - \ub \cdot \nabla \Bb]
	+ \curlGb \bigg)d\Omega \br
	&\quad 
	+ \frac{\mathcal{F}(R_j)}{E_j}\bigg[ -\int_{\Omega} (\nuPert |\nabla \ub|^2 + s \gamPert |\nabla \Bb|^2 ) d\Omega + \int_{\Omega} (\fb \cdot \ub + s (\curlGb) \cdot \Bb ) d\Omega \br
	&\quad+ \int_{\Gamma} B_{S}( \ub, \Bb ) d\Gamma \bigg] \br
	&\quad
	+ \bigg[ 1 - \frac{\mathcal{F}(R_j)}{E_j}\bigg] \Bigg| \int_{\Omega} (\fb \cdot \ub + s (\curlGb) \cdot \Bb ) d\Omega + \int_{\Gamma} B_{S}( \ub, \Bb ) d\Gamma \Bigg|, \nonumber
\end{align}
As will be seen later, we consider this reformulation (including the addition of the terms within absolute value brackets) as a means of constructing numerical schemes that inherit unconditional stability with respect to the modified energy $\mathcal{F}(R_j)$ and guaranteed positivity of a computed scalar variable $\xi_j$ to be defined. \\

With Dirichlet boundary conditions, a Crank-Nicolson scheme for \ref{eq:MHD} becomes
\begin{algorithm}\label{AlgoCN}
Given $\boldsymbol{u}_j^n$, $\boldsymbol{B}_j^n$, $\qu{n}$ and $\pe{n}$, find $\boldsymbol{u}_j^{n+1}$, $\boldsymbol{B}_j^{n+1}$, $\qu{n+1}$ and $p_{j}^{n+1}$ satisfying
\begin{flalign}
	&\left(\frac{\ue{n+1}-\ue{n}}{\Delta t}\right) = - \xi_j \left(\ute\cdot\nabla\right)\ute + s\xi_j \left(\bte\cdot\nabla\right)\bte + \nuM\Delta \ue{n+1/2} \label{algpCN1} &&\\
	&\qquad\qquad\qquad\qquad+ \nuP\Delta \ute - \nabla \pe{n+1/2} + \fe{n+1/2}, &&\br
	&\nabla\cdot \ue{n+1} = 0, \label{udivfree}
	&&\\
	&\left(\frac{\Be{n+1}-\Be{n}}{\Delta t}\right) = \xi_j \left(\bte\cdot\nabla\right)\ute 
	- \xi_j\left(\ute\cdot\nabla\right)\bte + \gamm \Delta \Be{n+1/2} \label{algpCN2} &&\\
	&\qquad\qquad\qquad\qquad+ \gamP \Delta \bte - \nabla \lame{n+1/2} + \geNew{n+1/2},&&\br
	&\nabla\cdot \Be{n+1} = 0, \label{bdivfree}
	&&\\
	&\xi_j = \frac{\mathcal{F}(\ru{n+1})}{E(\ubar{n+1}, \bbar{n+1})},
	&&\\
	&E(\ubar{n+1}, \bbar{n+1}) = \frac{1}{2}\norm{\ubar{n+1}}^2 + \frac{s}{2}\norm{\bbar{n+1}}^2 + C_0,&&\\	
	&\frac{\mathcal{F}(\ru{n+1}) - \mathcal{F}(\ru{n})}{\Delta t} 
	=
	\int_{\Omega} \ue{n+1/2}\cdot\left(\frac{\ue{n+1}-\ue{n}}{\Delta t}\right) d\Omega + \int_{\Omega} s\Be{n+1/2}\cdot\left(\frac{\Be{n+1}-\Be{n}}{\Delta t}\right) d\Omega &&\\
	&\qquad- \int_{\Omega}\ue{n+1/2}\cdot \bigg[- \xi_j \left(\ute\cdot\nabla\right)\ute + s\xi_j \left(\bte\cdot\nabla\right)\bte &&\br
	&\qquad + \nuM\Delta \ue{n+1/2} + \nuP\Delta \ute  - \nabla \pe{n+1/2} + \fe{n+1/2}\bigg] d\Omega &&\br
	&\qquad -\int_{\Omega}s\Be{n+1/2}\cdot \bigg[\xi_j \left(\bte\cdot\nabla\right)\ute  - \xi_j\left(\ute\cdot\nabla\right)\bte &&\br
	&\qquad  + \gamm \Delta \Be{n+1/2} + \gamP \Delta \bte - \nabla \lame{n+1/2} + \geNew{n+1/2}\bigg] d\Omega &&\br
	&\qquad+\xi_j\bigg[ -\int_{\Omega}\left(\nu_j|\nabla\ubar{n+1/2}|^2 + s\gamma_j |\nabla\bbar{n+1/2}|^2 \right)d\Omega + \int_{\Omega} \fe{n+1/2}\cdot\ubar{n+1/2} d\Omega 
	&&\br
	&\qquad+ \int_{\Omega} s(\geNew{n+1/2})\cdot\bbar{n+1/2} d\Omega + \int_{\Gamma} B_{S}(\fbe{n+1/2}, \gbe{n+1/2}, \ubar{n+1/2}, \bbar{n+1/2})d\Gamma \bigg] &&\br
	&\qquad
	+ (1-\xi_j)\bigg| \int_{\Omega} \fe{n+1/2}\cdot\ubar{n+1/2} d\Omega+ \int_{\Omega} s(\geNew{n+1/2})\cdot\bbar{n+1/2} d\Omega &&\br
	&\qquad
	+ \int_{\Gamma} B_{S}(\fbe{n+1/2}, \gbe{n+1/2}, \ubar{n+1/2}, \bbar{n+1/2})d\Gamma \bigg|.\nonumber&&
\end{flalign}
Here $\ubar{n+1}$, $\ubar{n+3/2}$, $\bbar{n+1}$ and $\bbar{n+3/2}$ are second order approximations of $\ue{n+1}$, $\ue{n+3/2}$, $\Be{n+1}$, and $\Be{n+3/2}$ that will be defined later. 
\end{algorithm}

Again for Dirichlet boundary conditions, a BDF2 scheme is
\begin{algorithm}\label{AlgoBDF2}
	Given $\boldsymbol{u}_j^n$, $\boldsymbol{B}_j^n$, $\qu{n}$ and $\pe{n}$, find $\boldsymbol{u}_j^{n+1}$, $\boldsymbol{B}_j^{n+1}$, $\qu{n+1}$ and $p_{j}^{n+1}$ satisfying
	\begin{flalign}
		&\left(\frac{3\ue{n+1}-4\ue{n}+\ue{n-1}}{2\Delta t}\right) = - \xi_j \left(\utte\cdot\nabla\right)\utte \label{algpBDF1} 
		+ s\xi_j \left(\btte\cdot\nabla\right)\btte + \nuM\Delta \ue{n+1} &&\\
		&\qquad\qquad\qquad\qquad\qquad+ \nuP\Delta \utte - \nabla \pe{n+1} + \fe{n+1},&&\br
		&\nabla\cdot \ue{n+1} = 0, \label{udivfreeCN}
		&&\\
		&\left(\frac{3\Be{n+1}-4\Be{n}+\Be{n-1}}{2\Delta t}\right) = \xi_j \left(\btte\cdot\nabla\right)\utte \label{algpBDF2}
		- \xi_j\left(\utte\cdot\nabla\right)\btte + \gamm \Delta \Be{n+1} &&\\
		&\qquad\qquad\qquad\qquad\qquad+ \gamP \Delta \btte - \nabla \lame{n+1} + \geNew{n+1},&&\br
		&\nabla\cdot \Be{n+1} = 0, \label{bdivfreeCN}
		&&\\
		&\xi_j = \frac{\mathcal{F}(\rubdf{n+3/2})}{E(\ubar{n+3/2}, \bbar{n+3/2})},
		&&\\
		&E(\ubar{n+3/2}, \bbar{n+3/2}) = \frac{1}{2}\norm{\ubar{n+3/2}}^2 + \frac{s}{2}\norm{\bbar{n+3/2}}^2 + C_0,&&\\	
		&\frac{\mathcal{F}(\rubdf{n+3/2}) - \mathcal{F}(\rubdf{n+1/2})}{\Delta t} 
		=
		\int_{\Omega} \ue{n+1}\cdot\left(\frac{3\ue{n+1}-4\ue{n}+\ue{n-1}}{2\Delta t}\right) d\Omega &&\br
		&\qquad+ \int_{\Omega} s\Be{n+1}\cdot\left(\frac{3\Be{n+1}-4\Be{n}+\Be{n-1}}{2\Delta t}\right) d\Omega &&\\
		&\qquad- \int_{\Omega}\ue{n+1}\cdot \bigg[- \xi_j \left(\utte\cdot\nabla\right)\utte + s\xi_j \left(\btte\cdot\nabla\right)\btte &&\br
		&\qquad+ \nuM\Delta \ue{n+1} + \nuP\Delta \utte  - \nabla \pe{n+1} + \fe{n+1}\bigg] d\Omega &&\br
		&\qquad- \int_{\Omega}s\Be{n+1}\cdot \bigg[\xi_j \left(\btte\cdot\nabla\right)\utte - \xi_j\left(\utte\cdot\nabla\right)\btte &&\br
		&\qquad + \gamm \Delta \Be{n+1} + \gamP \Delta \btte - \nabla \lame{n+1} + \geNew{n+1}\bigg] d\Omega
		&&\br
		&\qquad+\xi_j\bigg[ -\int_{\Omega}\left(\nu_j|\nabla\ubar{n+1}|^2 + s\gamma_j |\nabla\bbar{n+1}|^2 \right)d\Omega + \int_{\Omega} \fe{n+1}\cdot\ubar{n+1} d\Omega 
		&&\br
		&\qquad+ \int_{\Omega} s(\geNew{n+1})\cdot\bbar{n+1} d\Omega + \int_{\Gamma} B_{S}(\fbe{n+1}, \gbe{n+1}, \ubar{n+1}, \bbar{n+1})d\Gamma \bigg] &&\br
		&\qquad
		+ (1-\xi_j)\bigg| \int_{\Omega} \fe{n+1}\cdot\ubar{n+1} d\Omega+ \int_{\Omega} s(\geNew{n+1})\cdot\bbar{n+1} d\Omega &&\br
		&\qquad
		+ \int_{\Gamma} B_{S}(\fbe{n+1}, \gbe{n+1}, \ubar{n+1}, \bbar{n+1})d\Gamma \bigg|.&&\nonumber
	\end{flalign}
Similarly $\ubar{n+1}$, $\ubar{n+1/2}$, $\bbar{n+1}$ and $\bbar{n+1/2}$ are second order approximations of $\ue{n+1}$, $\ue{n+1/2}$, $\Be{n+1}$, and $\Be{n+1/2}$ to be defined later.
\end{algorithm}

The rest of the paper is outlined here. Section 2 gives mathematical preliminaries and defines notation. In Section 3, we prove the long time stability of the proposed algorithm. Section 4 presents an efficient way to implement our numerical algorithm. Section 5 numerically tests the proposed algorithm and illustrates theoretical results. Final conclusions and future directions are discussed in Section 6.

\section{Notation and preliminaries}

Throughout this paper the $L^2(\Omega)$ norm of scalars, vectors, and tensors will be denoted by $\Vert \cdot\Vert$ with the usual $L^2$ inner product denoted by $(\cdot, \cdot)$.  $H^{k}(\Omega)$ is the Sobolev space $%
W_{2}^{k}(\Omega)$, with norm $\Vert\cdot\Vert_{k}$. For functions $v(x,t)$ defined on $(0,T)$, we define the norms, for
$1\leq m<\infty$,
\[
\| v \|_{\infty,k} \text{ }:=EssSup_{[0,T]}\| v(\cdot, t)\|_{k}\qquad \text{and}\qquad \|v\|_{m,k} \text{ }:= \Big(  \int_{0}^{T}\|v(\cdot, t)\|_{k}^{m}\, dt\Big)
^{1/m} \text{ .}%
\]
The function spaces we consider are:
\begin{align*}
	X:&=H_{0}^{1}(\Omega )^{d}=\left\lbrace v\in L^2(\Omega)^d: \nabla v\in L^2(\Omega)^{d\times d}\text{ and }v=0 \text{ on } \partial \Omega\right\rbrace,\\
	Q : &=L_{0}^{2}(\Omega )=\left\lbrace q\in L^2(\Omega): \int_{\Omega} q \text{ }dx=0\right\rbrace,\\
	V : &=\left\lbrace v\in X:  (\nabla\cdot v, q)=0, \forall q\in Q \right\rbrace.
\end{align*}

A weak formulation of the full MHD equations is: Find $\boldsymbol{u}_j: [0,T]\rightarrow X$, $p_j: [0,T]\rightarrow Q$, $\boldsymbol{B}_j :  [0,T] \rightarrow X$ and $\lambda_j: [0,T]\rightarrow Q$ satisfying
\begin{align}
	&\left( \boldsymbol{u}_{j,t}, \boldsymbol{v}\right) + \left( \boldsymbol{u}_j\cdot\nabla \boldsymbol{u}_j, \boldsymbol{v} \right) - s\left( \boldsymbol{B}_j\cdot\nabla \boldsymbol{B}_j, \boldsymbol{v} \right) + \nu_j\left(\nabla  \boldsymbol{u}_j, \nabla \boldsymbol{v}\right) - \left( p_j, \nabla \cdot \boldsymbol{v}\right) = \left( \boldsymbol{f}_j, \boldsymbol{v}\right), \quad \forall \boldsymbol{v}\in X, \nonumber\\
	&\left( \nabla \cdot \boldsymbol{u}_j, l\right)=0, \quad \forall l\in Q,\nonumber\\
	&\left( \boldsymbol{B}_{j,t}, \boldsymbol{\chi}\right) + \left( \boldsymbol{u}_j\cdot\nabla \boldsymbol{B}_j, \boldsymbol{\chi} \right) - \left( \boldsymbol{B}_j\cdot\nabla \boldsymbol{u}_j, \boldsymbol{\chi} \right) + \gamma_j\left(\nabla  \boldsymbol{u}_j, \nabla \boldsymbol{\chi}\right) - \left( \lambda_j, \nabla \cdot \boldsymbol{\chi}\right) = \left( \boldsymbol{g}_j, \boldsymbol{\chi}\right), \quad \forall \boldsymbol{\chi}\in X, \nonumber\\
	&\left( \nabla \cdot \boldsymbol{B}_j, \psi\right)=0, \quad \forall \psi\in Q.\nonumber
\end{align}

We denote conforming velocity, pressure, potential finite element spaces based on an edge to edge
triangulation ($d=2$) or tetrahedralization ($d=3$) of $\Omega $ with maximum element diameter $h$ by 
\begin{equation*}
	X_{h}\subset X\text{ }\text{, }Q_{h}\subset Q.
\end{equation*}%
We also assume the finite element spaces ($X_{h}$, $Q_{h}$) satisfy the usual discrete inf-sup /$ LBB^{h}$
condition for stability of the discrete pressure, see \cite{1989} for more on this condition. Taylor-Hood elements, e.g., \cite{brenner2007mathematical}, \cite{1989}, are one such choice used in the tests in Section \ref{numTestSec}. We define the standard explicitly skew-symmetric trilinear form
\[
	b^{\ast}(u,v,w):=\frac{1}{2}(u\cdot\nabla v,w)-\frac{1}{2}(u\cdot\nabla w,v) 
\]

The full discretization of the proposed partitioned ensemble algorithm with Crank-Nicolson scheme is
\begin{algorithm}\label{AlgoCND}
	Given $\un{n}$, $\bn{n}$, $\qun{n}$, $\pn{n}$ and $\lamn{n}$, find $\un{n+1}$, $\bn{n+1}$, $\qun{n+1}$, $\pn{n+1}$ and $\lamn{n+1}$ satisfying for any $\boldsymbol{v_h},\boldsymbol{\chi_h}\in X_h$ and $l_h, \psi_h\in Q_h$,
	\begin{flalign}
		&\left(\frac{\un{n+1}-\un{n}}{\Delta t}, \boldsymbol{v_h}\right) = -\xi_j \bs{\ut}{\ut}{\boldsymbol{v_h}} + s\xi_j\bs{\bt}{\bt}{\boldsymbol{v_h}} \label{algp1CND} &&\\
		&\qquad - \nuM \left(\nabla\un{n+1/2}, \nabla \boldsymbol{v_h}\right) - \nuP\left(\nabla \ut, \nabla \boldsymbol{v_h}\right) + \left(\pn{n+1/2}, \nabla\cdot \boldsymbol{v_h}\right) &&\br
		&\qquad  - \alpha h \left(\nabla( \un{n+1} - \un{n} ), \nabla \boldsymbol{v_h}\right) + \left(\fn{n+1/2}, \boldsymbol{v_h}\right),
		&&\br
		&\left(\nabla\cdot \un{n+1}, l_h\right) = 0, \label{udivfreeCND} 
		&&\\
		&\left(\frac{\bn{n+1}-\bn{n}}{\Delta t}, \boldsymbol{\chi_h}\right) = \xi_j \bs{\bt}{\ut}{\boldsymbol{\chi_h}} - \xi_j\bs{\ut}{\bt}{\boldsymbol{\chi_h}}\label{algp2CND} &&\\
		&\qquad - \gamm \left(\nabla \bn{n+1/2}, \nabla \boldsymbol{\chi_h}\right) - \gamP \left(\nabla \bt, \nabla \boldsymbol{\chi_h}\right) + \left( \lamn{n+1/2}, \nabla\cdot \boldsymbol{\chi_h} \right)&&\br
		&\qquad  - \alpha_M h \left(\nabla( \bn{n+1} - \bn{n} ), \nabla \boldsymbol{\chi_h}\right) + \left(\gn{n+1/2}, \boldsymbol{\chi_h}\right),&&\br
		&\left(\nabla\cdot\bn{n+1}, \psi_h\right) = 0, \label{bdivfreeCND} 
		&&\\
		&\xi_j = \frac{\mathcal{F}(\run{n+1})}{E(\ubn{n+1}, \bbn{n+1})},
		&&\\
		&E(\ubn{n+1}, \bbn{n+1}) = \frac{1}{2}\norm{\ubn{n+1}}^2 + \frac{s}{2}\norm{\bbn{n+1}}^2 + C_0,
		&&\\	
		&\frac{\mathcal{F}(\run{n+1}) - \mathcal{F}(\run{n})}{\Delta t} 
		=
		\left( \frac{\un{n+1}-\un{n}}{\Delta t}, \un{n+1/2} \right) + s\left( \frac{\bn{n+1}-\bn{n}}{\Delta t}, \bn{n+1/2} \right) \label{scalarEqCND}&&\\
		&\qquad+ \xi_j \bs{\ut}{\ut}{\un{n+1/2}} - s\xi_j \bs{\bt}{\bt}{\un{n+1/2}} + \nuM\norm{ \nabla \un{n+1/2} }^2 &&\br
		&\qquad+ \nuP\left( \nabla \ut, \nabla \un{n+1/2}\right)  - \left( \pn{n+1/2}, \nabla\cdot \un{n+1/2} \right) &&\br
		&\qquad + \alpha h \left(\nabla( \un{n+1} - \un{n} ), \nabla \boldsymbol{v_h}\right) - \left(\fn{n+1/2}, \un{n+1/2} \right)
		&&\br
		&\qquad - s\xi_j \bs{\bt}{\ut}{\bn{n+1/2}} + s\xi_j\bs{\ut}{\bt}{\bn{n+1/2}} + s\gamm \norm{ \nabla \bn{n+1/2} }^2 
		&&\br
		&\qquad+ s\gamP \left( \nabla \btt, \nabla \bn{n+1/2} \right) - s\left( \lamn{n+1/2}, \nabla\cdot \bn{n+1/2} \right) + s\alpha_M h \left(\nabla( \bn{n+1} - \bn{n} ), \nabla \bn{n+1/2}\right)
		&&\br
		&\qquad - s\left( \gn{n+1/2}, \bn{n+1/2} \right) + \xi_j\bigg[ -\int_{\Omega}\left(\nu_j|\nabla\ubn{n+1/2}|^2 + s\gamma_j |\nabla\bbn{n+1/2}|^2 \right)d\Omega + \int_{\Omega} \fn{n+1/2}\cdot\ubn{n+1/2} d\Omega 
		&&\br
		&\qquad+ \int_{\Omega} s(\gn{n+1/2})\cdot\bbn{n+1/2} d\Omega + \int_{\Gamma} B_{S}(\fn{n+1/2}, \gn{n+1/2}, \ubn{n+1/2}, \bbn{n+1/2})d\Gamma \bigg] &&\br
		&\qquad
		+ (1-\xi_j)\bigg| \int_{\Omega} \fn{n+1/2}\cdot\ubn{n+1/2} d\Omega + \int_{\Omega} s(\gn{n+1/2})\cdot\bbn{n+1/2} d\Omega &&\br
		&\qquad
		+ \int_{\Gamma} B_{S}(\fn{n+1/2}, \gn{n+1/2}, \ubn{n+1/2}, \bbn{n+1/2})d\Gamma \bigg|.&&\nonumber
	\end{flalign}
\end{algorithm}

The full discretization of the proposed partitioned ensemble algorithm with BDF2 scheme is
\begin{algorithm}\label{AlgoBDFD}
	Given $\un{n}$, $\bn{n}$, $\qun{n}$, $\pn{n}$ and $\lamn{n}$, find $\un{n+1}$, $\bn{n+1}$, $\qun{n+1}$, $\pn{n+1}$ and $\lamn{n+1}$ satisfying for any $\boldsymbol{v_h},\boldsymbol{\chi_h}\in X_h$ and $l_h, \psi_h\in Q_h$,
	\begin{flalign}
	&\left(\frac{3\un{n+1}-4\un{n}+\un{n-1}}{2\Delta t}, \boldsymbol{v_h}\right) \label{algp1BDFD} 
	 = -\xi_j \bs{\utt}{\utt}{\boldsymbol{v_h}} + s\xi_j \bs{\btt}{\btt}{\boldsymbol{v_h}} &&\\
	&\qquad - \nuM \left(\nabla\un{n+1}, \nabla \boldsymbol{v_h}\right) - \nuP\left(\nabla \utt, \nabla \boldsymbol{v_h}\right)  +  \left(\pn{n+1}, \nabla\cdot \boldsymbol{v_h}\right) &&\br
	&\qquad 
	- \alpha h \left(\nabla( 3\un{n+1} - 4\un{n} + \un{n-1} ), \nabla \boldsymbol{v_h}\right) + \left(\fn{n+1}, \boldsymbol{v_h}\right),&&\br
	&\left(\nabla\cdot \un{n+1}, l_h\right) = 0, \label{udivfreeBDFD} &&\\
	&\left(\frac{3\bn{n+1}-4\bn{n}+\bn{n-1}}{2\Delta t}, \boldsymbol{\chi_h}\right) = \xi_j \bs{\btt}{\utt}{\boldsymbol{\chi_h}} - \xi_j \bs{\utt}{\btt}{\boldsymbol{\chi_h}} \label{algp2BDFD} &&\\
	&\qquad - \gamm \left(\nabla \bn{n+1}, \nabla \boldsymbol{\chi_h}\right) - \gamP \left(\nabla \btt, \nabla \boldsymbol{\chi_h}\right)  + \left( \lamn{n+1}, \nabla\cdot \boldsymbol{\chi_h} \right) &&\br
	&\qquad - \alpha_M h \left(\nabla( 3\bn{n+1} - 4\bn{n} + \bn{n-1} ), \nabla \boldsymbol{\chi_h}\right) + \left(\gn{n+1}, \boldsymbol{\chi_h}\right),&&\br
	&\left(\nabla\cdot\bn{n+1}, \psi_h\right) = 0, \label{bdivfreeBDFD} 
	&&\\
	&\xi_j = \frac{\mathcal{F}(\runbdf{n+1})}{E(\ubn{n+1}, \bbn{n+1})}, \label{scalarBDFD}
	&&\\
	&E(\ubn{n+3/2}, \bbn{n+3/2}) = \frac{1}{2}\norm{\ubn{n+3/2}}^2 + \frac{s}{2}\norm{\bbn{n+3/2}}^2 + C_0, \label{energyBDFD}
	&&\\	
	&\frac{\mathcal{F}(\runbdf{n+3/2}) - \mathcal{F}(\runbdf{n+1/2})}{\Delta t} 
	=
	\left( \frac{3\un{n+1}-4\un{n}+\un{n-1}}{2\Delta t}, \un{n+1} \right) &&\br
	&\qquad+ s\left( \frac{3\bn{n+1}-4\bn{n}+\bn{n-1}}{2\Delta t}, \bn{n+1} \right) + \xi_j \bs{\utt}{\utt}{\un{n+1}} \label{scalarEqBDFD} &&\\
	&\qquad- s\xi_j \bs{\btt}{\btt}{\un{n+1}} + \nuM\norm{ \nabla \un{n+1} }^2 + \nuP\left( \nabla \utt, \nabla \un{n+1}\right)&&\br
	&\qquad  - \left( \pn{n+1}, \nabla\cdot \un{n+1} \right) + \alpha h \left(\nabla( 3\un{n+1} - 4\un{n} + \un{n-1} ), \nabla \boldsymbol{v_h}\right) - \left(\fn{n+1}, \un{n+1} \right)
	&&\br
	&\qquad - s\xi_j \bs{\btt}{\utt}{\bn{n+1}}  + s\xi_j\bs{\utt}{\btt}{\bn{n+1}} + s\gamm \norm{ \nabla \bn{n+1} }^2 + s\gamP \left( \nabla \btt, \nabla \bn{n+1} \right)
	&&\br
	&\qquad - s\left( \lamn{n+1}, \nabla\cdot \bn{n+1} \right) + s\alpha_M h \left(\nabla( 3\bn{n+1} - 4\bn{n} + \bn{n-1} ), \nabla \boldsymbol{\chi_h}\right) - s\left( \gn{n+1}, \bn{n+1} \right)
	&&\br
	&\qquad+\xi_j\bigg[ -\int_{\Omega}\left(\nu_j|\nabla\ubn{n+1}|^2 + s\gamma_j |\nabla\bbn{n+1}|^2 \right)d\Omega + \int_{\Omega} \fn{n+1}\cdot\ubn{n+1} d\Omega 
	&&\br
	&\qquad+ \int_{\Omega} s(\gn{n+1})\cdot\bbn{n+1} d\Omega + \int_{\Gamma} B_{S}(\fn{n+1}, \gn{n+1}, \ubn{n+1}, \bbn{n+1})d\Gamma \bigg] &&\br
	&\qquad
	+ (1-\xi_j)\bigg| \int_{\Omega} \fn{n+1}\cdot\ubn{n+1} d\Omega+ \int_{\Omega} s(\gn{n+1})\cdot\bbn{n+1} d\Omega &&\br
	&\qquad
	+ \int_{\Gamma} B_{S}(\fn{n+1}, \gn{n+1}, \ubn{n+1}, \bbn{n+1})d\Gamma \bigg|.&&\nonumber
	\end{flalign}
\end{algorithm}
There's also the addition of two regularization terms in Algorithms \eqref{AlgoCND} and \eqref{AlgoBDFD},

\begin{flalign*}
	& \begin{aligned} & 
		\begin{cases}
			\alpha h \Delta( \un{n+1} - \un{n} ), \\
			\alpha_M h \Delta( \bn{n+1} - \bn{n} ), \\
		\end{cases}\text{for CN,}
	\end{aligned} \qquad
	\begin{aligned}
		& 
		\begin{cases}
			\alpha h \Delta( 3\un{n+1} - 4\un{n} + \un{n-1} ), \\
			\alpha_M h \Delta( 3\bn{n+1} - 4\bn{n} + \bn{n-1} ), \\
		\end{cases} \text{for BDF2.}
	\end{aligned}
\end{flalign*}
These terms are highly effective at reducing the considerable error that eventually appears when the timestep is not sufficiently refined. Significant improvement in accuracy  will be seen later in the numerical tests. It's noted in \cite{LLMNR2009} that this improvement cannot be explained by the stability or error analysis alone. Instead, an explanation is offered through analysis of a modified form of the equations under consideration. In the modified equations, the addition of the term $-\alpha h k \Delta u_t$ (in the case of velocity) and $-\alpha h k \Delta B_t$ (in the case of magnetic field) are added to the left-hand sides,
\begin{equation}\label{eq:MHDmod}
	\left\{\begin{aligned}
	&\left[\boldsymbol{u}_{j,t} - \alpha h k\Delta\boldsymbol{u}_{j,t}\right] +\boldsymbol{u}_j\cdot\nabla \boldsymbol{u}_j  -s\boldsymbol{B}_j\cdot \nabla \boldsymbol{B}_j - \nu_j \Delta \boldsymbol{u}_j + \nabla p_j  
	=\boldsymbol{f}_j \text{ in } \Omega\times (0,T),\\
	&
	\nabla \cdot \boldsymbol{u}_j = 0, \text{ in } \Omega\times (0,T),\\
	&
	\left[\boldsymbol{B}_{j,t} - s\alpha_M h k \Delta\boldsymbol{B}_{j,t} \right] + \boldsymbol{u}_j\cdot \nabla \boldsymbol{B}_j - \boldsymbol{B}_j\cdot \nabla \boldsymbol{u}_j - \gamma_j \Delta \boldsymbol{B}_j + \nabla \lambda_j
	=
	\nabla \times \boldsymbol{g}_j \text{ in } \Omega\times (0,T),\\
	&
	\nabla \cdot \boldsymbol{B}_j = 0, \text{ in } \Omega\times (0,T),\\
	&
	\boldsymbol{u}_j(x,0)   =\boldsymbol{u}_j^{0}(x),   \text{ in } \Omega,\quad \boldsymbol{B}_j(x,0)    = \boldsymbol{B}_j^{0}(x), \text{ in } \Omega.
	\end{aligned}\right.
\end{equation}
This results in a modified kinetic energy corresponding to the equation. In our case, the resulting modified kinetic energy would be
\[
	\norm{u(t)}^2 + \alpha h k \norm{\nabla u(t)}^2 + s\norm{B(t)}^2 + s \alpha_M h k\norm{\nabla B(t)}^2.
\]
Following Kraichnan's theory \cite{Kraichnan1967}, it is argued in \cite{ LLMNR2009} that the penalty term in the kinetic energy induces an enhanced energy decay rate for numerically under-resolved modes while preserving the correct energy cascade above the cut-off length scale.  The quick roll-off  in the energy spectrum  is  also exploited in the Navier-Stokes-$\alpha$ model (NS-$\alpha$)--a nonlinearly dispersive modification of the Navier-Stokes equations for large eddy simulation of turbulence  \cite{FHT2001, CHMZ1999}.  This roll-off mechanism shortens the inertial range  and makes the system more computable.

\section{Stability of the method}

\subsection{Crank-Nicolson}\label{stability}
\begin{theorem}
	With homogeneous boundary conditions and forcing terms equal to zero, Algorithm \eqref{AlgoCND} is unconditionally stable with respect to the modified energy $\mathcal{F}(R_j)$.\\
\end{theorem}

\begin{proof}
Stability follows directly from \cite{YaDo2020}. Set $\boldsymbol{v_h}$ to $\un{n+1/2}$ in \eqref{algp1CND}, $\boldsymbol{\chi_h}$ to $s\bn{n+1/2}$ in \eqref{algp2CND}, add each of these to \eqref{scalarEqCND} and note \eqref{udivfreeCND} and \eqref{bdivfreeCND}. Then one gets
\begin{align}
	\mathcal{F}(\run{n+1}) - \mathcal{F}(\run{n}) = -\Delta t \frac{ \mathcal{F}(\run{n+1}) }{ E(\ubn{n+1}, \bbn{n+1}) } \int_{\Omega}\left(\nu_j |\nabla\ubn{n+1/2}|^2 + s\gamma_j |\nabla\bbn{n+1/2}|^2 \right)d\Omega \\
	+ \left[ 1 - \frac{ \mathcal{F}(\run{n+1}) }{ E(\ubn{n+1}, \bbn{n+1}) } \right]|S_0|\Delta t + \frac{ \mathcal{F}(\run{n+1}) }{ E(\ubn{n+1}, \bbn{n+1}) } S_0\Delta t. \nonumber
\end{align}
Where $S_0 = \int_{\Omega} \fn{n+1/2}\cdot\ubn{n+1/2} d\Omega + \int_{\Omega} s(\gn{n+1/2})\cdot\bbn{n+1/2} d\Omega$. Solving for $\mathcal{F}( \run{n+1} )$ gives
\begin{align}
	\mathcal{F}(\run{n+1}) = \frac{ \mathcal{F}(\run{n}) + |S_0|\Delta t }{ 1 + \frac{\Delta t}{E(\ubn{n+1}, \bbn{n+1}) } \left[\int_{\Omega}\left(\nu_j |\nabla\ubn{n+1/2}|^2 + s\gamma_j |\nabla\bbn{n+1/2}|^2 \right)d\Omega + (|S_0|-S_0) \right] }.
\end{align}
If $\fb = 0$ and $ \curlGb = 0$, then $S_0 = 0$ and
\begin{align}\label{fTermCN}
	\mathcal{F}(\run{n+1}) = \frac{ \mathcal{F}(\run{n}) }{ 1 + \frac{\Delta t}{E(\ubn{n+1}, \bbn{n+1}) } \int_{\Omega}\left(\nu_j |\nabla\ubn{n+1/2}|^2 + s\gamma_j |\nabla\bbn{n+1/2}|^2 \right)d\Omega }.
\end{align}
Note the denominator in \eqref{fTermCN} is greater than or equal to $1$. By definition \eqref{incFuncF}, if $\run{0} > 0$, then $\mathcal{F}(\run{0}) > 0$. In fact $\run{0}$ would be initialized as $\mathcal{G}(E[\boldsymbol{u}_j^0(x), \boldsymbol{B}_j^0(x)])$, which by definition \eqref{incFuncG} is guaranteed positive. Then by induction for any timestep $n$, $\mathcal{F}(\run{n+1}) > 0$, giving us
\begin{align}\label{Fineq}
	0 < \mathcal{F}(\run{n+1}) \leq \mathcal{F}(\run{n}),\qquad n\geq 0.
\end{align}
This completes the proof.
\end{proof}

\subsection{BDF2}\label{stabilityBDF}
\begin{theorem}
	With homogeneous boundary conditions and forcing terms equal to zero, Algorithm \eqref{AlgoBDFD} is unconditionally stable with respect to the modified energy $\mathcal{F}(R_j)$ as long as the approximations of $R_j(t)$ at timestep $\frac{1}{2}$ are positive.\\
\end{theorem}

\begin{proof}
	If one sets $\boldsymbol{v_h}$ to $\un{n+1}$ in \eqref{algp1BDFD} and $\boldsymbol{\chi_h}$ to $s\bn{n+1}$ in \eqref{algp2BDFD}, subtracts each of these from \eqref{scalarEqBDFD} and notes \eqref{udivfreeBDFD} and \eqref{bdivfreeBDFD}, the proof follows identically to \cite{YaDo2020}. We have
	\begin{align}
		\mathcal{F}(\runbdf{n+3/2}) - \mathcal{F}(\runbdf{n+1/2}) = -\Delta t \frac{ \mathcal{F}(\runbdf{n+3/2}) }{ E(\ubn{n+3/2}, \bbn{n+3/2}) } \int_{\Omega}\left(\nu_j |\nabla\ubn{n+1}|^2 + s\gamma_j |\nabla\bbn{n+1}|^2 \right)d\Omega \\
		+ \left[ 1 - \frac{ \mathcal{F}(\runbdf{n+3/2}) }{ E(\ubn{n+3/2}, \bbn{n+3/2}) } \right]|S_0|\Delta t + \frac{ \mathcal{F}(\runbdf{n+3/2}) }{ E(\ubn{n+3/2}, \bbn{n+3/2}) } S_0\Delta t. \nonumber
	\end{align}
	Where $S_0 = \int_{\Omega} \fn{n+1}\cdot\ubn{n+1} d\Omega + \int_{\Omega} s(\gn{n+1})\cdot\bbn{n+1} d\Omega$. Solving for $\mathcal{F}( \run{n+3/2} )$ gives
	\begin{align}
		\mathcal{F}(\runbdf{n+3/2}) = \frac{ \mathcal{F}(\runbdf{n+1/2}) + |S_0| }{1 + \frac{\Delta t}{ E(\ubn{n+3/2}, \bbn{n+3/2}) } [ \int_{\Omega}\left(\nu_j |\nabla\ubn{n+1}|^2 + s\gamma_j |\nabla\bbn{n+1}|^2 \right)d\Omega + (|S_0| - S_0) ] }.
	\end{align}
	If $\fb = 0$ and $ \curlGb = 0$, then $S_0 = 0$ and
	\begin{align}\label{fTermBDF}
		\mathcal{F}(\runbdf{n+3/2}) = \frac{ \mathcal{F}(\runbdf{n+1/2}) }{1 + \frac{\Delta t}{ E(\ubn{n+3/2}, \bbn{n+3/2}) } \int_{\Omega}\left(\nu_j |\nabla\ubn{n+1}|^2 + s\gamma_j |\nabla\bbn{n+1}|^2 \right)d\Omega }.
	\end{align}
	The denominator above is greater than or equal to $1$. Now by definition \eqref{incFuncF}, if it's ensured the approximation of $R_j(t)$ at timestep $1/2$ is positive, i.e. $\runbdf{1/2} > 0$, then $\mathcal{F}(\runbdf{1/2}) > 0$. Then by induction for any timestep $n$, $\mathcal{F}(\run{n+3/2}) > 0$ and
	\begin{align}\label{FineqBDF}
		0 < \mathcal{F}(\run{n+3/2}) \leq \mathcal{F}(\run{n+1/2}),\qquad n\geq 0.
	\end{align}
	This completes the proof.
\end{proof}

Note that for the choice of $\mathcal{F}(\chi) = \chi^2 \geq 0$ for all $\chi \in (-\infty, \infty)$, \eqref{FineqBDF} and unconditional stability will hold regardless of whether $\runbdf{1/2} > 0$.\\

\section{Implementation}
Since the schemes are linear and the auxiliary variables are scalar functions of time variable, the resulting systems can be solved conveniently by superposition of a series of Stokes-type equations. We illustrate the idea by presenting the algorithms in strong form.
\subsection{Crank-Nicolson}
To efficiently implement Algorithm \eqref{AlgoCN}, we proceed in the following manner. Assume
\begin{align*}
	\ue{n+1} &= \uhe{n+1} + \xi_j\uue{n+1}, \qquad \pe{n+1} = \phe{n+1} + \xi_j \pue{n+1}, \\
	\Be{n+1} &= \bhe{n+1} + \xi_j\bue{n+1}, \qquad \lame{n+1} = \lamhe{n+1} + \xi_j \lamue{n+1}.
\end{align*}
Then solving Algorithm \eqref{AlgoCN} is equivalent to solving the following subproblems,
\begin{algorithm}\label{effAlgCN} Given $\boldsymbol{u}_j^n$, $\boldsymbol{B}_j^n$ and $\pe{n}$,\\
	Sub-problem 1: find $\uhe{n+1}$, $\bhe{n+1}$, $\phe{n+1}$ and $\lamhe{n+1}$ satisfying
	\begin{subequations}\label{sbp1}
	\begin{align}
		&\frac{1}{\Delta t} \uhe{n+1} - \frac{\nuM}{2} \Delta \uhe{n+1}  + \frac{1}{2}\nabla \phe{n+1} = \fe{n+1/2} + \frac{1}{\Delta t} \ue{n} + \nuP \Delta \ute \\
		&\qquad + \frac{\nuM}{2} \Delta \ue{n} - \frac{1}{2}\nabla \pe{n}, \quad \nabla \cdot \uhe{n+1} = 0,\br
		&\frac{1}{\Delta t}\bhe{n+1} - \frac{\gamm}{2}\Delta \bhe{n+1} + \frac{1}{2}\nabla \lamhe{n+1} = \geNew{n+1/2} + \frac{1}{\Delta t}\Be{n} + \frac{\gamm}{2}\Delta \Be{n} \\
		&\qquad  + \gamP \Delta \bte - \frac{1}{2}\nabla \lame{n},
		\quad \nabla \cdot \bhe{n+1} = 0,\nonumber
	\end{align}
	\end{subequations}
	Sub-problem 2: find $\uue{n+1}$, $\bue{n+1}$, $\pue{n+1}$ and $\lamue{n+1}$ satisfying
	\begin{subequations}\label{sbp2}
	\begin{align}
	&\frac{1}{\Delta t} \uue{n+1} - \frac{\nuM}{2} \Delta \uue{n+1} + \frac{1}{2}\nabla \pue{n+1}=s\left(\bte\cdot\nabla\right)\bte - \left(\ute\cdot\nabla\right)\ute, \\
	& \nabla \cdot \uue{n+1} = 0,\\	
	& \frac{1}{\Delta t}\bue{n+1} + \frac{1}{2}\nabla \lamue{n+1} - \frac{\gamm}{2}\Delta \bue{n+1}
	 =  \left(\bte\cdot\nabla\right)\ute 
	 - \left(\ute\cdot\nabla\right)\bte, \\
	 & \nabla \cdot \bue{n+1} = 0.
	\end{align}
	\end{subequations}
\end{algorithm}
\begin{remark}
	For inhomogeneous Dirichlet boundary conditions, let
	\[
		\uhe{n+1} = g(x,t^{n+1}), \quad \uue{n+1} = 0, \quad \bhe{n+1} = h(x,t^{n+1}), \quad \bue{n+1} = 0 \quad \text{on } \partial\Omega.
	\]
\end{remark}
We use the following approximations,
\begin{numcases}{}
	\vhbar{n+1} = \vhe{n+1} + \vue{n+1}, \\
	\vhbar{n+1/2} = \frac{1}{2}(\vhbar{n+1} + \vh{n}).
\end{numcases}
We then update $\xi_j$ as
\begin{align}\label{scalarCN}
	\xi_j = \frac{\mathcal{F}(\ru{n}) + |S_0|\Delta t}{ E(\ubar{n+1}, \bbar{n+1}) + \Delta t \int_{\Omega}\left(\nu|\nabla\ubar{n+1/2}|^2 + s\gamma |\nabla\bbar{n+1/2}|^2 \right)d\Omega + \Delta t (|S_0|-S_0)},
\end{align}
where
\begin{align}
	S_0 = \int_{\Omega} \fe{n+1/2}\cdot\ubar{n+1/2} d\Omega+ \int_{\Omega} s(\geNew{n+1/2})\cdot\bbar{n+1/2} d\Omega + \int_{\Gamma} B_{S}( \ubar{n+1/2}, \bbar{n+1/2} ) d\Gamma.
\end{align}
Notice $\xi_j$ is updated via a linear equation and is very direct. Once we have $\xi_j$ we update
\begin{align}\label{compRCN}
	\ru{n+1} = \mathcal{G}\left( \xi_j E(\ubar{n+1}, \bbar{n+1}) \right)
\end{align}
and proceed to the next timestep iteration. Since $\xi_j$ is a ratio of the SAV to itself, we should expect the result to be close to one. With our ensemble approach in \eqref{sbp1}-\eqref{sbp2}, all $J$ realizations have the same coefficient matrix in each timestep so should be computationally efficient.\\

\begin{theorem}\label{posInCN}
	The scalar $\xi_j$ in \eqref{scalarCN} and $\ru{n+1}$ in \eqref{compRCN} are guaranteed to be positive at all timesteps.
\end{theorem}
\begin{proof}
	By definition \eqref{incFuncF}, $\mathcal{F}(\ru{0}) > 0$ so long as $\ru{0} > 0$. It's explained in \eqref{stability} that $\ru{0}$ will be positive. The energy function $E(u, B)$ is always positive and $\int_{\Omega}\left(\nu|\nabla u|^2 + s\gamma |\nabla B|^2 \right)d\Omega \geq 0$. Since $|S_0| - S_0 \geq 0$, the initially computed $\xi_j$ is ensured positive. Then by induction, $\xi_j$ at any timestep is guaranteed positive.
	
	Once it's ensured $\xi_j > 0$, from the definition \eqref{incFuncG} we can guarantee $\ru{n+1}$ in \eqref{compRCN} is positive. This completes the proof.
\end{proof}

\subsection{BDF2}
For Algorithm \eqref{AlgoBDF2}, we develop an efficient implementation with the same approach. Note solving Algorithm \eqref{AlgoBDF2} is equivalent to the following,
\begin{algorithm}\label{effAlgBDF2} Given $\boldsymbol{u}_j^n$, $\boldsymbol{B}_j^n$ and $\pe{n}$,\\
	Sub-problem 1: find $\uhe{n+1}$, $\bhe{n+1}$, $\phe{n+1}$ and $\lamhe{n+1}$ satisfying
	\begin{subequations}\label{sbp1BDF2}
	\begin{align}
	\frac{3}{2\Delta t} \uhe{n+1} - \nuM \Delta \uhe{n+1} + \nabla \phe{n+1} = \fe{n+1} + \frac{2}{\Delta t} \ue{n} - \frac{1}{2\Delta t} \ue{n-1} + \nuP \Delta, \\
	\nabla \cdot \uhe{n+1} = 0, \\
	\frac{3}{2\Delta t}\bhe{n+1} - \gamm\Delta \bhe{n+1} + \nabla \lamhe{n+1} = \geNew{n+1} + \frac{2}{\Delta t}\Be{n} - \frac{1}{2\Delta t} \Be{n-1} + \gamP \Delta,\\
	\nabla \cdot \bhe{n+1} = 0,
	\end{align}
	\end{subequations}
	Sub-problem 2: find $\uue{n+1}$, $\bue{n+1}$, $\pue{n+1}$ and $\lamue{n+1}$ satisfying
	\begin{subequations}\label{sbp2BDF2}
	\begin{align}
	&\frac{3}{2\Delta t} \uue{n+1} - \nuM \Delta \uue{n+1} + \nabla \pue{n+1} =s\left(\btte\cdot\nabla\right)\btte - \left(\utte\cdot\nabla\right)\utte, \\
	&\nabla \cdot \uhe{n+1} = 0,\\
	& \frac{3}{2\Delta t}\bue{n+1} - \gamm\Delta \bue{n+1} + \nabla \lamue{n+1}  =   \left(\btte\cdot\nabla\right)\utte  - \left(\utte\cdot\nabla\right)\btte,\\
	& \nabla \cdot \bhe{n+1} = 0.
	\end{align}
	\end{subequations}
\end{algorithm}
We use the following approximations,
\begin{numcases}{}
	\vhbarbdf{n+1} = \vhe{n+1} + \vue{n+1}, \\
	\vhbarbdf{n+3/2} = \frac{3}{2}\vhbarbdf{n+1} - \frac{1}{2}\vh{n}.
\end{numcases}
We update $\xi_j$ as
\begin{align}\label{scalarBDF}
	\xi_j = \frac{\mathcal{F}(\rubdf{n+1/2}) + |S_0|\Delta t}{ E(\ubar{n+3/2}, \bbar{n+3/2}) + \Delta t\int_{\Omega}\left(\nu|\nabla\ubar{n+1}|^2 + s\gamma |\nabla\bbar{n+1}|^2 \right)d\Omega + \Delta t (|S_0|-S_0)},
\end{align}
where
\[
	S_0 = \int_{\Omega} \fe{n+1}\cdot\ubar{n+1} d\Omega+ \int_{\Omega} s(\geNew{n+1})\cdot\bbar{n+1} d\Omega + \int_{\Gamma} B_{S}( \ubar{n+1}, \bbar{n+1} ) d\Gamma.
\]
Once we have $\xi_j$ we update $\ru{n+1}$ as follows:
\begin{numcases}{}
	\rubdf{n+3/2} = \mathcal{G}\left( \xi_j E(\ubar{n+3/2}, \bbar{n+3/2}) \right), \label{compRBDF1}\\
	\ru{n+1} = \frac{2}{3}\rubdf{n+3/2} + \frac{1}{3}\ru{n}. \label{compRBDF2}
\end{numcases}
and proceed to the next timestep iteration.\\

\begin{theorem}
	The scalar $\xi_j$ in \eqref{effAlgBDF2} and $\ru{n+1}$ in \eqref{compRBDF2} are guaranteed to be positive at all timesteps if the approximation $\rubdf{1/2} > 0$.\\
\end{theorem}
\begin{proof}
	Again by definition \eqref{incFuncF}, $\mathcal{F}(\rubdf{1/2}) > 0$ so long as approximation $\rubdf{1/2} > 0$. The argument for positivity of $\xi_j$ proceeds identically to that made in the proof of Theorem \eqref{posInCN}.
	
	Once it's ensured $\xi_j > 0$, again from definition \eqref{incFuncG} we can guarantee $\rubdf{n+3/2}$ in \eqref{compRBDF1} is positive. It's also guaranteed $\ru{0}$ is positive from the previously stated point that it would be initialized as $\mathcal{G}(E(\boldsymbol{u}_j^0(x), \boldsymbol{B}_j^0(x)))$. Thus we conclude $\ru{n+1}$ in \eqref{compRBDF2} remains positive. This completes the proof.
\end{proof}

\section{Numerical tests}\label{numTestSec}
This section will present numerical results for Algorithms \eqref{AlgoCND} and \eqref{AlgoBDFD} to demonstrate the expected convergence rates and the stability proven previously. We set $\mathcal{F}(\chi) = \chi^2$ and the corresponding $\mathcal{G}(\chi) = \sqrt{\chi}$ in every experiment. Throughout these tests we'll use the finite element triplet $(P^2–P^1–P^2)$, and the finite element software package FEniCS \cite{FENICS}.\\

\subsection{Convergence Test}\label{convT} To verify the expected convergence rates, we will use a variation of the test problem in \cite{decoupMHD}. Take the time interval $0\leq t \leq 1$ and domain $\Omega=[0,1]^2$. Define the true solution $(u,p,B)$ as
\[
	\begin{cases}
		u_{\epsilon} = \left( y^5 + t^2 , x^5 + t^2 \right)(1+\epsilon),\\
		p_{\epsilon} = 10(2x-1)(2y-1)(1+t^2)(1+\epsilon),\\
		B_{\epsilon} = \left( \sin{(\pi y)} + t^2, \sin{(\pi x)} + t^2 \right)(1+\epsilon),
	\end{cases}
\]
where $\epsilon$ is a given perturbation. For this problem we will consider two perturbations $\epsilon_1 = 10^{-1}$ and $\epsilon_2 = -10^{-1}$. The kinematic viscosity and magnetic resistivity are defined as $\nu_{\epsilon} = 0.5 \cdot (1+\epsilon)$ and $\gamma_{\epsilon} = 0.5 \cdot (1+\epsilon)$. The source terms and initial conditions correspond with the exact solution for the given perturbation. The results are displayed in tables \eqref{tableu1conv}-\eqref{tableB2conv} both with regularization and without ($\alpha = \alpha_M=0)$.

\begin{table}[H]
	\centering
	\caption{Crank-Nicolson error and convergence rates for the first ensemble member in $u_h$ and $\nabla u_h$.}
	\label{tableu1conv}
	\begin{tabular}{|c|c|c|c|c|c|}
		\hline
		\rowcolor{lightgray}
		h & $\Delta t$ & $\norm{u_1-u_{1,h}}_{\infty,0}$ & Rate & $\norm{\nabla u_1-\nabla u_{1,h}}_{2,0}$ & Rate \\
		\hline
		1/10 & 1/8 & 9.191 e-4 & --- & 4.985 e-3 & --- \\  
		1/20 & 1/16 & 2.088 e-4 & 2.138 & 1.399 e-3 & 1.834\\
		1/40 & 1/32 &  4.810 e-5 & 2.118 & 3.679 e-4 & 1.927\\
		1/80 & 1/64 & 1.154 e-5 & 2.060 & 9.422 e-5 & 1.965\\
		1/160 & 1/128 & 2.889 e-6 & 1.998 & 2.384 e-5 & 1.983
		\\
		\hline
		\rowcolor{lightgray}
		\multicolumn{6}{l}{\hfil \qquad \qquad \qquad \qquad \quad Reg with  $\alpha = \alpha_M = 0.5$} \\
		\hline
		1/10 & 1/8 & 3.912 e-4 & --- & 4.741 e-3 & --- \\  
		1/20 & 1/16 & 6.032 e-5 & 2.697 & 1.355 e-3 & 1.807\\
		1/40 & 1/32 &  9.532 e-6 & 2.662 & 3.579 e-4 & 1.920\\
		1/80 & 1/64 & 2.208 e-6 & 2.110 & 9.179 e-5 & 1.963
		\\
		\hline
	\end{tabular}
\end{table}
\begin{table}[H]
	\centering
	\caption{Crank-Nicolson error and convergence rates for the first ensemble member in $B_h$ and $\nabla B_h$.}
	\label{tableB1conv}
	\begin{tabular}{|c|c|c|c|c|c|}
		\hline
		\rowcolor{lightgray}
		h & $\Delta t$ & $\norm{B_1 - B_{1,h}}_{\infty,0}$ & Rate & $\norm{\nabla B_1 - \nabla B_{1,h}}_{2,0}$ & Rate \\
		\hline
		1/10 & 1/8 & 2.566 e-4 & --- & 3.013 e-3 & ---\\  
		1/20 & 1/16 & 5.0568 e-5 & 2.343 & 8.451 e-4 & 1.834\\
		1/40 & 1/32 &  1.150 e-5 & 2.136 & 2.223 e-4 & 1.927\\
		1/80 & 1/64 & 2.746 e-6 & 2.067 & 5.694 e-5 & 1.965\\
		1/160 & 1/128 & 6.869 e-7 & 1.999 & 1.440 e-5 & 1.983
		\\
		\hline
		\rowcolor{lightgray}
		\multicolumn{6}{l}{\hfil \qquad \qquad \qquad \qquad \quad Reg with  $\alpha = \alpha_M = 0.5$} \\
		\hline
		1/10 & 1/8 & 1.512 e-4 & --- & 2.909 e-3 & --- \\  
		1/20 & 1/16 & 2.138 e-5 & 2.822 & 8.298 e-4 & 1.810\\
		1/40 & 1/32 &  3.082 e-6 & 2.795 & 2.191 e-4 & 1.921\\
		1/80 & 1/64 & 6.830 e-7 & 2.174 & 5.619 e-5 & 1.964
		\\
		\hline        
	\end{tabular}
\end{table}
\begin{table}[H]
	\centering
	\caption{Crank-Nicolson error and convergence rates for the second ensemble member in $u_h$ and $\nabla u_h$.}
	\label{tableu2conv}
	\begin{tabular}{|c|c|c|c|c|c|}
		\hline
		\rowcolor{lightgray}
		h & $\Delta t$ & $\norm{u_2-u_{2,h}}_{\infty,0}$ & Rate & $\norm{\nabla u_2-\nabla u_{2,h}}_{2,0}$ & Rate \\
		\hline
		1/10 & 1/8 & 2.020 e-3 & --- & 5.498 e-3 & ---\\  
		1/20 & 1/16 & 4.897 e-4 & 2.045 & 1.433 e-3 & 1.940\\
		1/40 & 1/32 & 9.342 e-5 & 2.390 & 3.701 e-4 & 1.953\\
		1/80 & 1/64 & 1.560 e-5 & 2.582 & 9.440 e-5 & 1.971\\
		1/160 & 1/128 & 2.923 e-6 & 2.416 & 2.385 e-5 & 1.985
		\\
		\hline
		\rowcolor{lightgray}
		\multicolumn{6}{l}{\hfil \qquad \qquad \qquad \qquad \quad Reg with  $\alpha = \alpha_M = 0.5$} \\
		\hline
		1/10 & 1/8 & 4.070 e-4 & --- & 4.753 e-3 & --- \\  
		1/20 & 1/16 & 6.277 e-5 & 2.697 & 1.357 e-3 & 1.809\\
		1/40 & 1/32 & 1.134 e-5 & 2.469 & 3.584 e-4 & 1.921\\
		1/80 & 1/64 & 2.649 e-6 & 2.097 & 9.190 e-5 & 1.964
		\\
		\hline          
	\end{tabular}
\end{table}
\begin{table}[H]
	\centering
	\caption{Crank-Nicolson error and convergence rates for the second ensemble member in $B_h$ and $\nabla B_h$.}
	\label{tableB2conv}
	\begin{tabular}{|c|c|c|c|c|c|}
		\hline
		\rowcolor{lightgray}
		h & $\Delta t$ & $\norm{B_2-B_{2,h}}_{\infty,0}$ & Rate & $\norm{\nabla B_2-\nabla B_{2,h}}_{2,0}$ & Rate \\
		\hline
		1/10 & 1/8 & 7.455 e-4 & --- & 3.376 e-3 & ---\\  
		1/20 & 1/16 & 1.666 e-4 & 2.162 & 8.700 e-4 & 1.956\\
		1/40 & 1/32 & 3.097 e-5 & 2.427 & 2.239 e-4 & 1.958\\
		1/80 & 1/64 & 5.113 e-6 & 2.598 & 5.706 e-5 & 1.973\\
		1/160 & 1/128 & 7.772 e-7 & 2.718 & 1.442 e-5 & 1.985
		\\
		\hline
		\rowcolor{lightgray}
		\multicolumn{6}{l}{\hfil \qquad \qquad \qquad \qquad \quad Reg with  $\alpha = \alpha_M = 0.5$} \\
		\hline
		1/10 & 1/8 & 1.567 e-4 & --- & 2.915 e-3 & --- \\  
		1/20 & 1/16 & 2.222 e-5 & 2.818 & 8.308 e-4 & 1.811\\
		1/40 & 1/32 &  3.664 e-6 & 2.600 & 2.193 e-4 & 1.922\\
		1/80 & 1/64 & 8.188 e-7 & 2.162 & 5.622 e-5 & 1.964
		\\
		\hline    
	\end{tabular}
\end{table}

\begin{table}[H]
	\centering
	\caption{BDF2 error and convergence rates for the first ensemble member in $u_h$ and $\nabla u_h$.}
	\label{tableu1convBDF2}
	\begin{tabular}{|c|c|c|c|c|c|}
		\hline
		\rowcolor{lightgray}
		h & $\Delta t$ & $\norm{u_1-u_{1,h}}_{\infty,0}$ & Rate & $\norm{\nabla u_1-\nabla u_{1,h}}_{2,0}$ & Rate \\
		\hline
		1/10 & 1/8 & 7.413 e-4 & --- & 5.804 e-3 & --- \\  
		1/20 & 1/16 & 1.891 e-4 & 1.971 & 1.495 e-3 & 1.957\\
		1/40 & 1/32 & 4.790 e-5 & 1.981 & 3.793 e-4 & 1.978\\
		1/80 & 1/64 & 1.183 e-5 & 2.018 & 9.557 e-5 & 1.989\\
		1/160 & 1/128 & 2.944 e-6 & 2.006 & 2.399 e-5 & 1.994
		\\
		\hline    
		\rowcolor{lightgray}
		\multicolumn{6}{l}{\hfil \qquad \qquad \qquad \qquad \quad Reg with  $\alpha = \alpha_M = 0.5$} \\
		\hline
		1/10 & 1/8 & 4.528 e-4 & --- & 5.601 e-3 & --- \\  
		1/20 & 1/16 & 6.215 e-5 & 2.865 & 1.453 e-3 & 1.947\\
		1/40 & 1/32 &  7.946 e-6 & 2.968 & 3.694 e-4 & 1.976\\
		1/80 & 1/64 & 1.339 e-6 & 2.570 & 9.310 e-5 & 1.988
		\\
		\hline          
	\end{tabular}
\end{table}
\begin{table}[H]
	\centering
	\caption{BDF2 error and convergence rates for the first ensemble member in $B_h$ and $\nabla B_h$.}
	\label{tableB1convBDF2}
	\begin{tabular}{|c|c|c|c|c|c|}
		\hline
		\rowcolor{lightgray}
		h & $\Delta t$ & $\norm{B_1 - B_{1,h}}_{\infty,0}$ & Rate & $\norm{\nabla B_1 - \nabla B_{1,h}}_{2,0}$ & Rate \\
		\hline
		1/10 & 1/8 & 1.868 e-4 & --- & 3.502 e-3 & ---\\  
		1/20 & 1/16 & 3.792 e-5 & 2.301 & 9.005 e-4 & 1.960\\
		1/40 & 1/32 &  9.133 e-6 & 2.054 & 2.285 e-4 & 1.979\\
		1/80 & 1/64 & 2.300 e-6 & 1.990 & 5.756 e-5 & 1.989\\
		1/160 & 1/128 & 5.816 e-7 & 1.983 & 1.445 e-5 & 1.994
		\\
		\hline       
		\rowcolor{lightgray}
		\multicolumn{6}{l}{\hfil \qquad \qquad \qquad \qquad \quad Reg with  $\alpha = \alpha_M = 0.5$} \\
		\hline
		1/10 & 1/8 & 1.649 e-4 & --- & 3.438 e-3 & --- \\  
		1/20 & 1/16 & 2.185 e-5 & 2.916 & 8.904 e-4 & 1.949\\
		1/40 & 1/32 &  2.772 e-6 & 2.978 & 2.263 e-4 & 1.976\\
		1/80 & 1/64 & 4.182 e-7 & 2.729 & 5.705 e-5 & 1.988
		\\
		\hline       
	\end{tabular}
\end{table}
\begin{table}[H]
	\centering
	\caption{BDF2 error and convergence rates for the second ensemble member in $u_h$ and $\nabla u_h$.}
	\label{tableu2convBDF2}
	\begin{tabular}{|c|c|c|c|c|c|}
		\hline
		\rowcolor{lightgray}
		h & $\Delta t$ & $\norm{u_2-u_{2,h}}_{\infty,0}$ & Rate & $\norm{\nabla u_2-\nabla u_{2,h}}_{2,0}$ & Rate \\
		\hline
		1/10 & 1/8 & 7.762 e-4 & --- & 5.806 e-3 & ---\\  
		1/20 & 1/16 & 1.880 e-4 & 2.045 & 1.495 e-3 & 1.957 \\
		1/40 & 1/32 & 4.699 e-5 & 2.001 & 3.795 e-4 & 1.978 \\
		1/80 & 1/64 & 1.186 e-5 & 1.987 & 9.561 e-5 & 1.989 \\
		1/160 & 1/128 & 2.964 e-6 & 2.001 & 2.400 e-5 & 1.994 
		\\
		\hline         
		\rowcolor{lightgray}
		\multicolumn{6}{l}{\hfil \qquad \qquad \qquad \qquad \quad Reg with  $\alpha = \alpha_M = 0.5$} \\
		\hline
		1/10 & 1/8 & 4.531 e-4 & --- & 5.603 e-3 & --- \\  
		1/20 & 1/16 & 6.218 e-5 & 2.865 & 1.453 e-3 & 1.947\\
		1/40 & 1/32 &  7.964 e-6 & 2.965 & 3.695 e-4 & 1.976\\
		1/80 & 1/64 & 1.547 e-6 & 2.364 & 9.314 e-5 & 1.988
		\\
		\hline     
	\end{tabular}
\end{table}
\begin{table}[H]
	\centering
	\caption{BDF2 error and convergence rates for the second ensemble member in $B_h$ and $\nabla B_h$.}
	\label{tableB2convBDF2}
	\begin{tabular}{|c|c|c|c|c|c|}
		\hline
		\rowcolor{lightgray}
		h & $\Delta t$ & $\norm{B_2-B_{2,h}}_{\infty,0}$ & Rate & $\norm{\nabla B_2-\nabla B_{2,h}}_{2,0}$ & Rate \\
		\hline
		1/10 & 1/8 & 1.918 e-4 & --- & 3.505 e-3 & ---\\  
		1/20 & 1/16 & 3.930 e-5 & 2.287 & 9.013 e-4 & 1.960\\
		1/40 & 1/32 & 9.605 e-6 & 2.033 & 2.287 e-4 & 1.979\\
		1/80 & 1/64 & 2.425 e-6 & 1.986 & 5.761 e-5 & 1.989\\
		1/160 & 1/128 & 6.129 e-7 & 1.984 & 1.446 e-5 & 1.994
		\\
		\hline        
		\rowcolor{lightgray}
		\multicolumn{6}{l}{\hfil \qquad \qquad \qquad \qquad \quad Reg with  $\alpha = \alpha_M = 0.5$} \\
		\hline
		1/10 & 1/8 & 1.649 e-3 & --- & 3.439 e-3 & --- \\  
		1/20 & 1/16 & 2.185 e-4 & 2.916 & 8.906 e-4 & 1.949\\
		1/40 & 1/32 &  2.772 e-4 & 2.978 & 2.264 e-4 & 1.976\\
		1/80 & 1/64 & 4.880 e-5 & 2.506 & 5.706 e-5 & 1.988
		\\
		\hline      
	\end{tabular}
\end{table}

\subsection{Stability}
Here we analyze the stability of the second order ensemble methods. For the test problem, we will exclude external energy and body forces so that in observation if the method is stable, the system energy should decay to zero as time passes. We also use the initial conditions,
\[
	\begin{cases}
		u_{\epsilon}^0 = ( x^2(x-1)^2 y(y-1)(2y - 1)(1+\epsilon), -y^2(y-1)^2 x(x-1)(2x - 1)(1+\epsilon) ), \\
		p_{\epsilon}^0 = 0,\\
		B_{\epsilon}^0 = ( \sin{(\pi x)}\cos{(\pi y)}(1+\epsilon), -\sin{(\pi y)}\cos{(\pi x)} ).
	\end{cases}
\]
We fix the coupling term $s = 1$ and choose two different sets of viscosity and magnetic viscosity to test, $\nu = \gamma = 0.1$ and $\nu = \gamma = 0.02$. The mesh discretization is fixed at $h = 1/50$ and several time steps are employed, with final time $T = 5$.

\begin{figure}[H] 
  \subfloat[Decay of total system energy to $T=5$ for Algorithm \eqref{AlgoCND} with $\nu = \gamma = 0.1$.\label{CN_stability_fig1}]{%
    \includegraphics[scale=0.45]{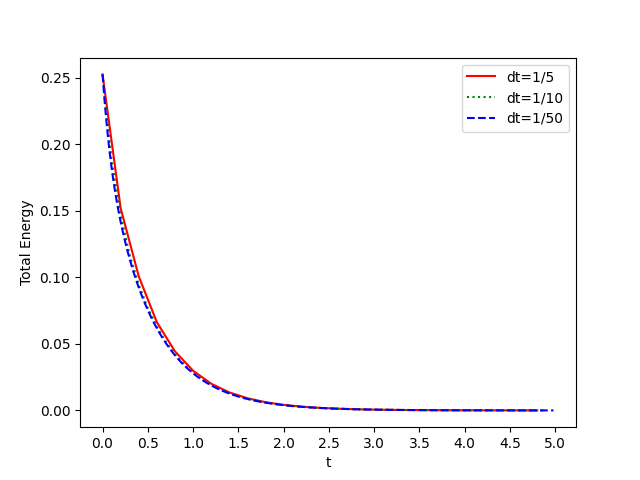} 
  } 
  \qquad\qquad
  \subfloat[Decay of total system energy to $T=5$ for Algorithm \eqref{AlgoCND} with $\nu = \gamma = 0.02$.\label{CN_stability_fig2}]{%
    \includegraphics[scale=0.45]{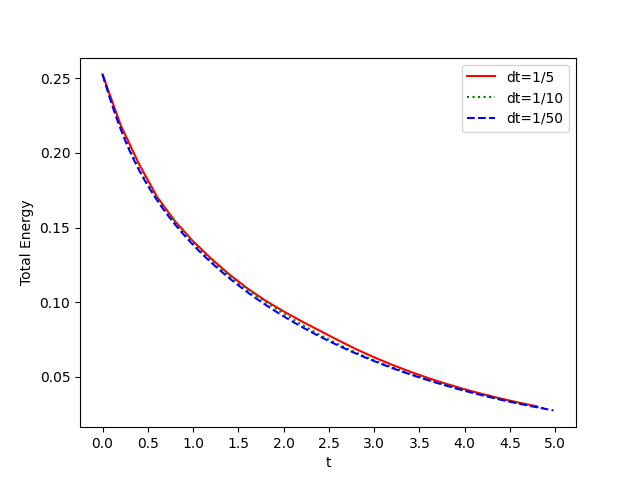} 
  }
\end{figure}

\begin{figure}[H] 
  \subfloat[Decay of total system energy to $T=5$ for Algorithm \eqref{AlgoBDFD} with $\nu = \gamma = 0.1$.\label{BDF_stability_fig1}]{%
    \includegraphics[scale=0.45]{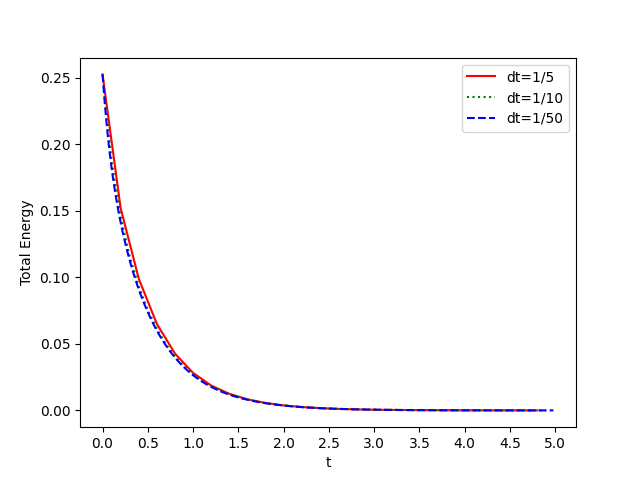} 
  } 
  \qquad\qquad
  \subfloat[Decay of total system energy to $T=5$ for Algorithm \eqref{AlgoBDFD} with $\nu = \gamma = 0.02$.\label{BDF_stability_fig2}]{%
    \includegraphics[scale=0.45]{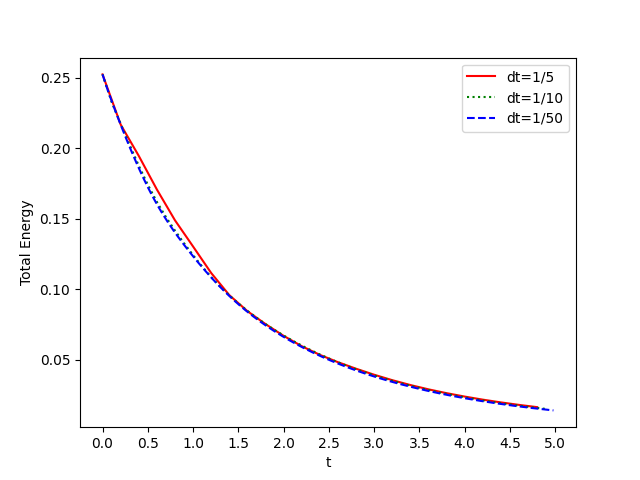} 
  }
\end{figure}

\subsection{Chamber Flow}
In this numerical test, we consider a channel flow in a rectangular domain of length 2.2 units and height 0.41, with a cylinder of radius $0.05$ centered at $(0.2, 0.2)$, in the presence of a magnetic field. On the walls and around the cylinder, a no-slip boundary condition is applied for velocity while magnetic field is kept constant as $B = < 0, 0.1 >^T$. We set the inflow and outflow conditions equal, choosing $u = < 6 y (0.41-y)/°0.41^2 \sin{(\pi t/16.0)}, 0 >^T$ and $B = <0,0.1>^T$. The coupling term is set to $s = 0.01$ and for all realizations we fix $\gamma = 0.1$ then consider two cases, $\nu = 1/50$ and $\nu = 1/1000$.

We'll use an ensemble of two different solutions with the initial and boundary conditions perturbed by multiplicative factors of $(1 \pm \epsilon)$. We simulate the flow with Algorithms \eqref{AlgoCND} and \eqref{AlgoBDFD} till final time $T= 8.8$ with a mesh discretization fixed at $h = 1/100$. We set $\alpha = \alpha_M = 0$ such that these tests are performed without the regularization terms involved. In order to maintain accurate results up unto $T=8.8$, we find it necessary to choose a time step of roughly $\Delta t = 1/1000$ when $\nu = 1/50$ and $\Delta t = 1/2000$ when $\nu = 1/1000$. The solutions under each perturbation for velocity are shown in (\ref{CN_flow_fig1})-(\ref{BDF2_flow_fig2}) and for magnetic field in Figures (\ref{CN_mag_fig1})-(\ref{BDF2_mag_fig2}). We also provide results of a traditional scheme with no perturbation, i.e. $\epsilon = 0$,
\begin{subequations} \label{primitive}
\begin{align}
		\frac{1}{\Delta t} \left(\un{n+1}, \boldsymbol{v_h}\right) + \nu \left(\nabla \un{n+1}, \boldsymbol{v_h}\right) + \alpha h \left(\nabla \un{n+1}, \nabla \boldsymbol{v_h}\right) \label{primS1}\\
		\qquad + \bs{\un{n}}{\un{n+1}}{\boldsymbol{v_h}} -s\bs{\bn{n}}{\bn{n+1}}{\boldsymbol{v_h}} - \left(\pn{n+1}, \nabla \cdot \boldsymbol{v_h} \right) = \frac{1}{\Delta t}\left(\un{n},\boldsymbol{v_h}\right) \br
		\qquad + \alpha h \left(\nabla\un{n}, \nabla \boldsymbol{v_h}\right) + \left(\fn{n+1}, \boldsymbol{v_h}\right), 
		\br
		\left(\nabla\cdot \un{n+1}, l_h\right) = 0, \label{primS2}
		\\
		\frac{1}{\Delta t} \left(\bn{n+1}, \boldsymbol{\chi_h}\right) + \gamma \left(\nabla \bn{n+1}, \nabla \boldsymbol{\chi_h}\right) + \alpha_M h \left(\nabla \bn{n+1}, \nabla \boldsymbol{\chi_h}\right) \label{primS3}\\
		\qquad + \bs{\un{n}}{\bn{n+1}}{\boldsymbol{\chi_h}} -\bs{\bn{n}}{\un{n+1}}{\boldsymbol{\chi_h}} - \left(\lamn{n+1}, \nabla \cdot \boldsymbol{\chi_h} \right) = \frac{1}{\Delta t}\left(\bn{n},\boldsymbol{\chi_h}\right) \br
		\qquad + \alpha_M h \left(\nabla\bn{n}, \nabla \boldsymbol{\chi_h}\right) + \left(\gn{n+1}, \boldsymbol{\chi_h}\right),  \br
		\left(\nabla\cdot\bn{n+1}, \psi_h\right) = 0.\label{primS4}
\end{align}
\end{subequations}
for comparison.

\begin{figure}[H] \captionsetup[subfigure]{labelformat=empty}
  \subfloat[]{%
    \includegraphics[scale=0.55]{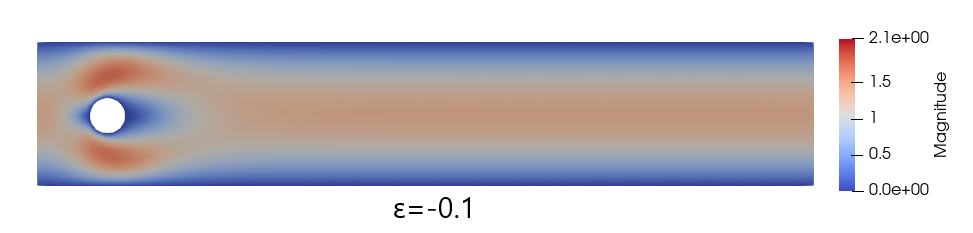} 
  } 
  \qquad\qquad
  \subfloat[]{%
    \includegraphics[scale=0.55]{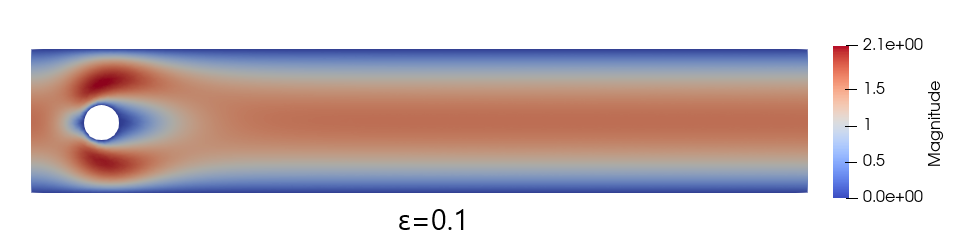} 
  } \\
  \caption{ Ensemble solutions for velocity at time $T=8.8$ for Algorithm \eqref{AlgoCND} with $\nu = 0.02$, $\gamma = 0.1$ and $\Delta t = 0.001$.\label{CN_flow_fig1} }
\end{figure}

\begin{figure}[H] \captionsetup[subfigure]{labelformat=empty}
  \subfloat[]{%
    \includegraphics[scale=0.45]{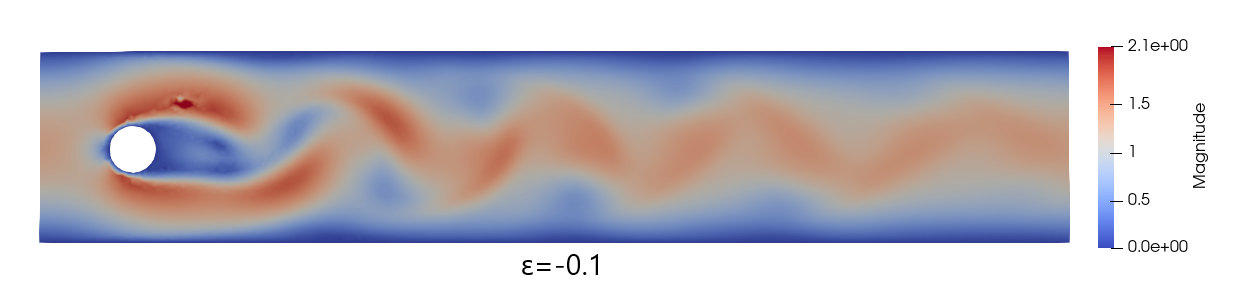} 
  } 
  \qquad\qquad
  \subfloat[]{%
    \includegraphics[scale=0.45]{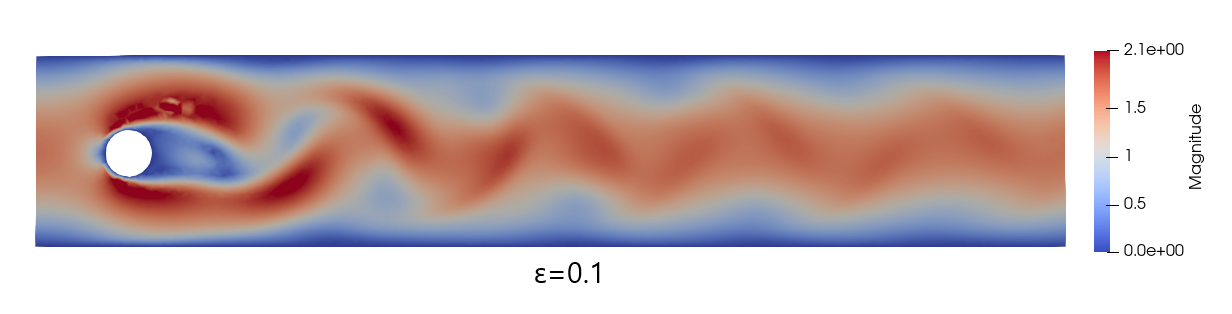} 
  } \\
  \caption{ Ensemble solutions for velocity at time $T=8.8$ for Algorithm \eqref{AlgoCND} with $\nu = 0.001$, $\gamma = 0.1$ and $\Delta t = 0.0005$.\label{CN_flow_fig2} }
\end{figure}

\begin{figure}[H] \captionsetup[subfigure]{labelformat=empty}
  \subfloat[]{%
    \includegraphics[scale=0.555]{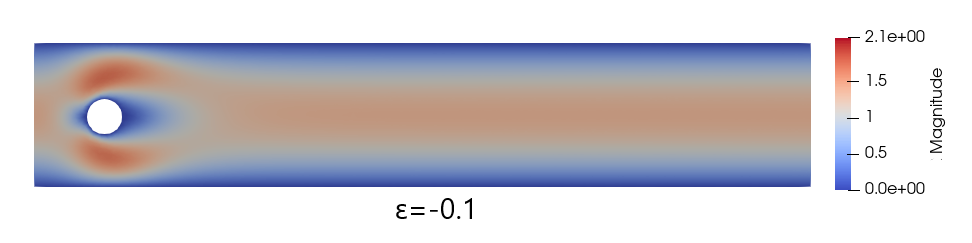} 
  } 
  \qquad\qquad
  \subfloat[]{%
    \includegraphics[scale=0.555]{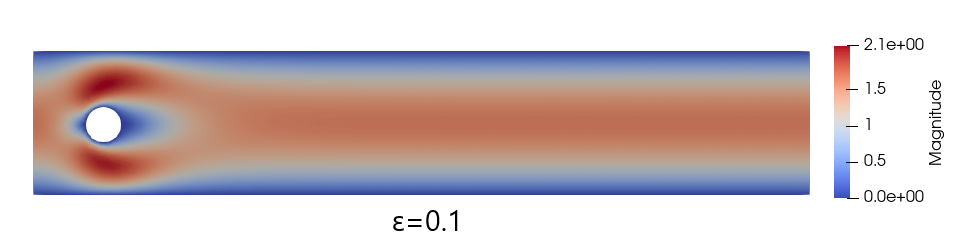} 
  } \\
  \caption{ Ensemble solutions for velocity at time $T=8.8$ for Algorithm \eqref{AlgoBDFD} with $\nu = 0.02$, $\gamma = 0.1$ and $\Delta t = 0.001$.\label{BDF2_flow_fig1} }
\end{figure}

\begin{figure}[H] \captionsetup[subfigure]{labelformat=empty}
  \subfloat[]{%
    \includegraphics[scale=0.575]{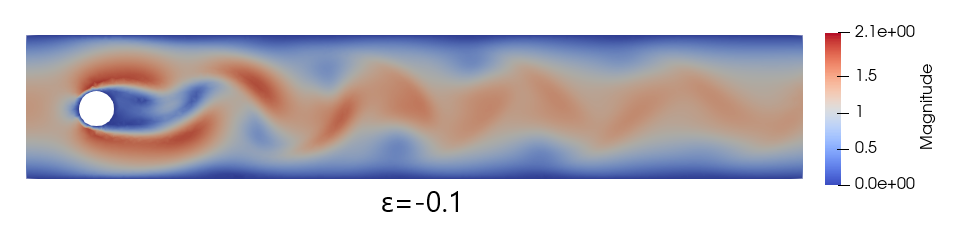} 
  } 
  \qquad\qquad
  \subfloat[]{%
    \includegraphics[scale=0.575]{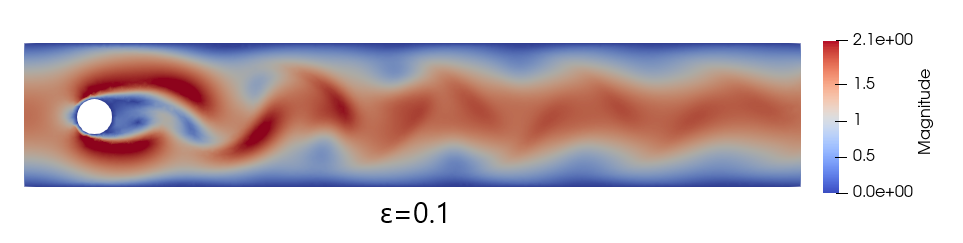} 
  }
  \caption{ Ensemble solutions for velocity at time $T=8.8$ for Algorithm \eqref{AlgoBDFD} with $\nu = 0.001$, $\gamma = 0.1$ and $\Delta t = 0.0005$.\label{BDF2_flow_fig2} \\}
\end{figure}
\begin{figure}
  \vspace{10mm}
  \subfloat[]{%
    \includegraphics[scale=0.46]{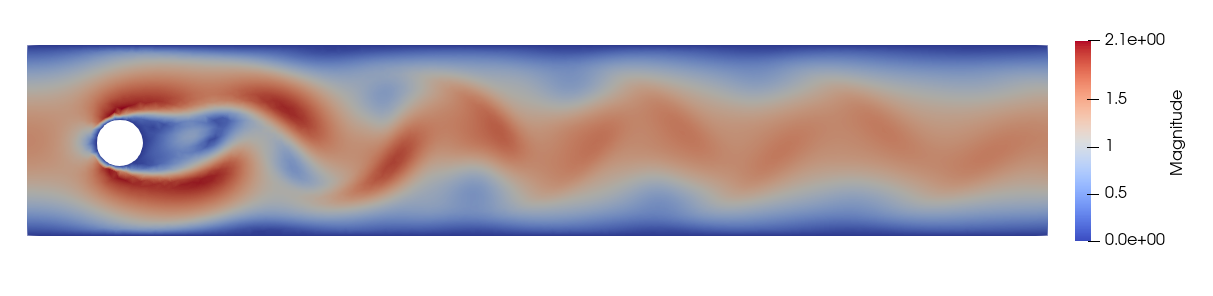} 
  } \\
  \caption{ Ensemble solutions for velocity at time $T=8.8$ for primitive scheme \eqref{primitive} with $\nu = 0.001$, $\gamma = 0.1$ and $\Delta t = 0.001$.\label{Prim_flow_fig2} }
\end{figure}


\begin{figure}[H] \captionsetup[subfigure]{labelformat=empty}
  \subfloat[]{%
    \includegraphics[scale=0.59]{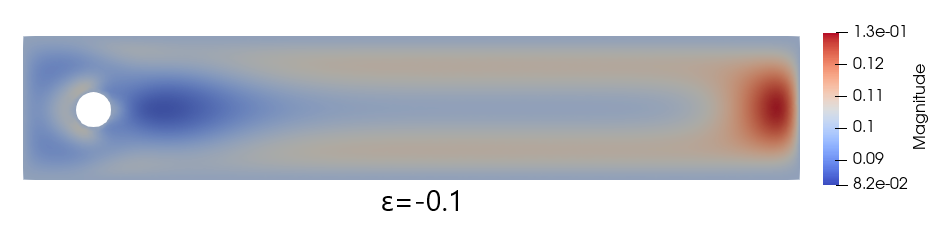} 
  } 
  \qquad\qquad
  \subfloat[]{%
    \includegraphics[scale=0.59]{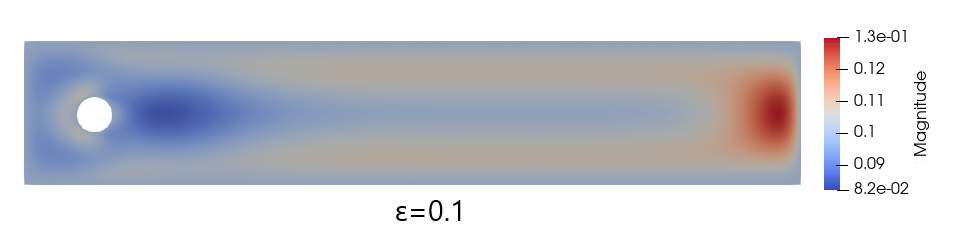} 
  } \\
  \caption{ Ensemble solutions for magnetic field at time $T=8.8$ for Algorithm \eqref{AlgoCND} with $\nu = 0.02$, $\gamma = 0.1$ and $\Delta t = 0.001$.\label{CN_mag_fig1} }
\end{figure}

\begin{figure}[H] \captionsetup[subfigure]{labelformat=empty}
  \subfloat[]{%
    \includegraphics[scale=0.44]{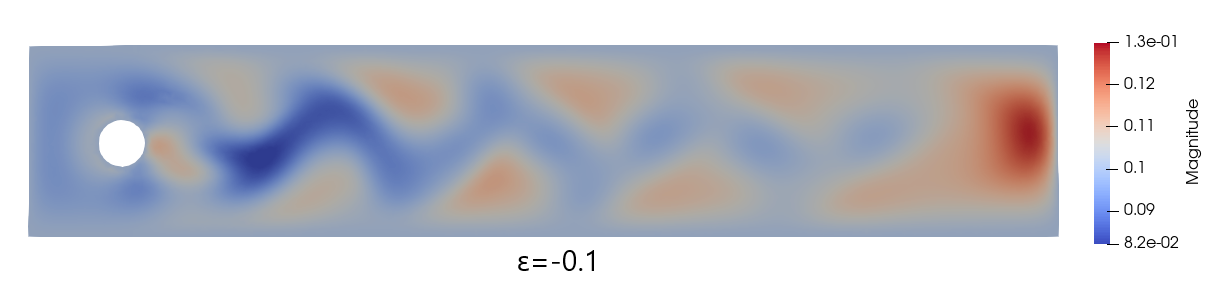} 
  } 
  \qquad\qquad
  \subfloat[]{%
    \includegraphics[scale=0.44]{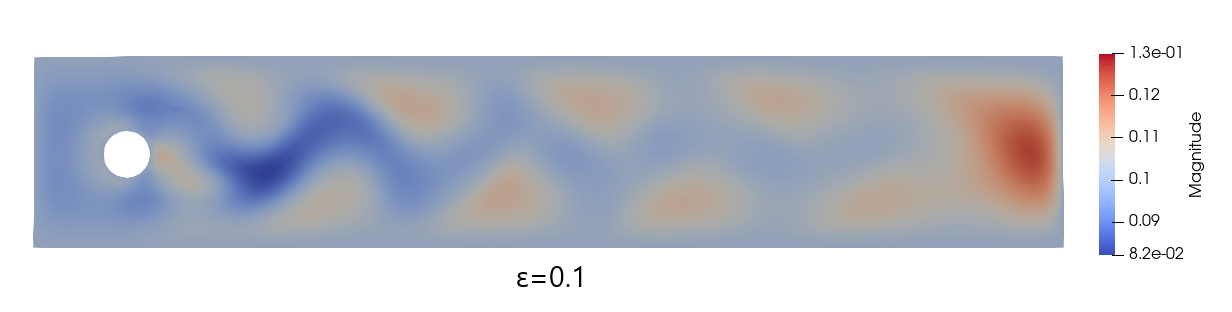} 
  } \\
  \caption{ Ensemble solutions for magnetic field at time $T=8.8$ for Algorithm \eqref{AlgoCND} with $\nu = 0.001$, $\gamma = 0.1$ and $\Delta t = 0.0005$.\label{CN_mag_fig2} }
\end{figure}

\begin{figure}[H] \captionsetup[subfigure]{labelformat=empty}
  \subfloat[]{%
    \includegraphics[scale=0.57]{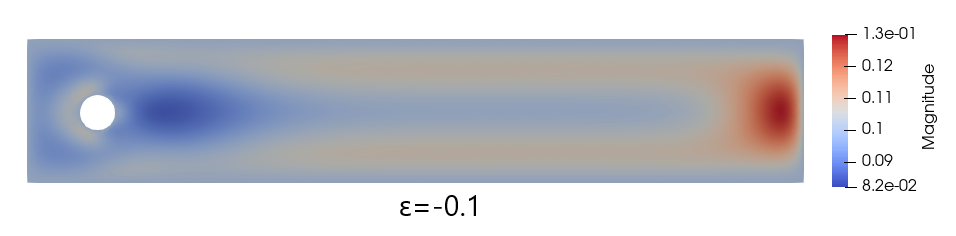} 
  } 
  \qquad\qquad
  \subfloat[]{%
    \includegraphics[scale=0.56]{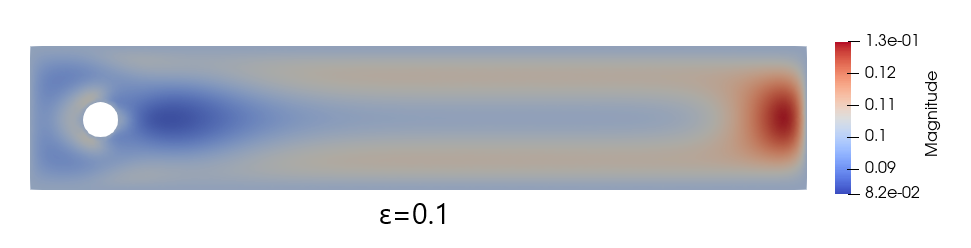} 
  } \\
  \caption{ Ensemble solutions for magnetic field at time $T=8.8$ for Algorithm \eqref{AlgoBDFD} with $\nu = 0.02$, $\gamma = 0.1$ and $\Delta t = 0.001$.\label{BDF2_mag_fig1} }
\end{figure}

\begin{figure}[H] \captionsetup[subfigure]{labelformat=empty}
  \subfloat[]{%
    \includegraphics[scale=0.575]{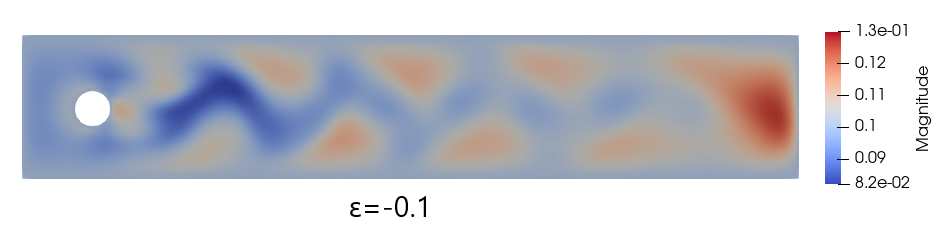} 
  } 
  \qquad\qquad
  \subfloat[]{%
    \includegraphics[scale=0.575]{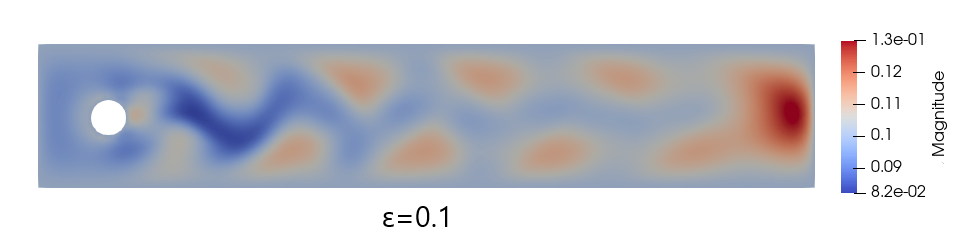} 
  } \\
  \caption{ Ensemble solutions for magnetic field at time $T=8.8$ for Algorithm \eqref{AlgoBDFD} with $\nu = 0.001$, $\gamma = 0.1$ and $\Delta t = 0.0005$.\label{BDF2_mag_fig2} }
\end{figure}


\begin{figure}[H] \captionsetup[subfigure]{labelformat=empty}
  \subfloat[]{%
    \includegraphics[scale=0.47]{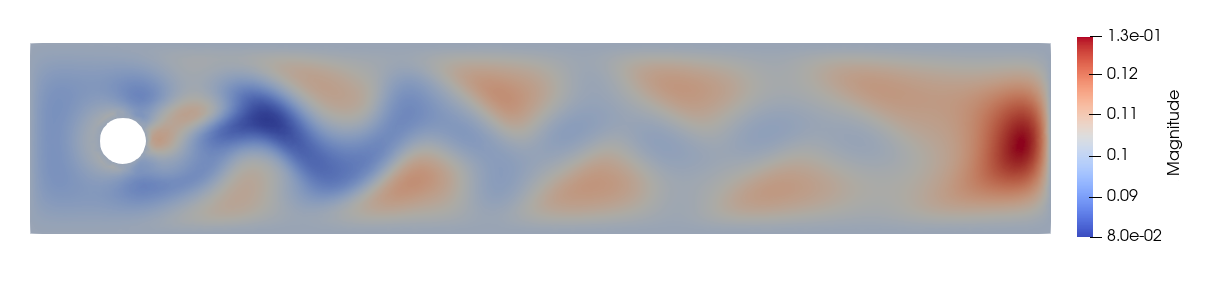} 
  } \\
  \caption{ Ensemble solutions for magnetic field at time $T=8.8$ for primitive scheme \eqref{primitive} with $\nu = 0.001$, $\gamma = 0.1$ and $\Delta t = 0.001$.\label{Prim_mag_fig} }
\end{figure}


\subsection{Chamber Flow with Regularization}\label{ChFlow}
Here we present the same chamber flow problem implementing Algorithms \eqref{AlgoCND} and \eqref{AlgoBDFD} with nonzero regularization coefficients. We choose $\alpha = \nu$ and $\alpha_M = \gamma$ in each test. We're able to achieve similar accuracy to the previous section with coarser time step. The following numerical results are achieved:

\begin{figure}[H] \captionsetup[subfigure]{labelformat=empty}
  \subfloat[]{%
    \includegraphics[scale=0.45]{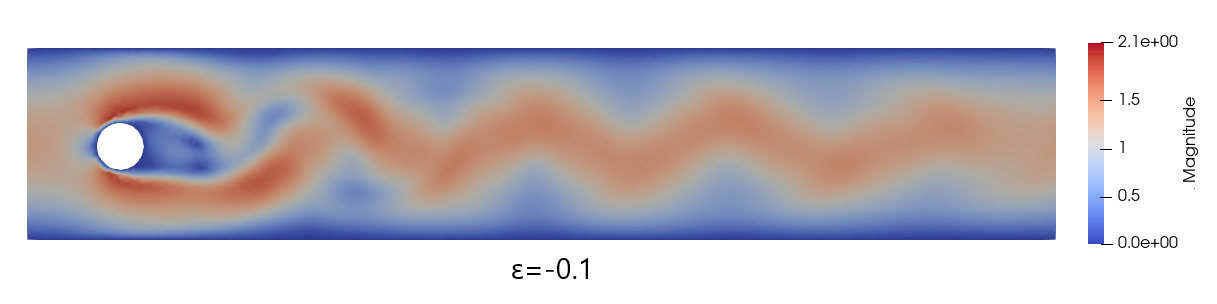} 
  } 
  \qquad\qquad
  \subfloat[]{%
    \includegraphics[scale=0.45]{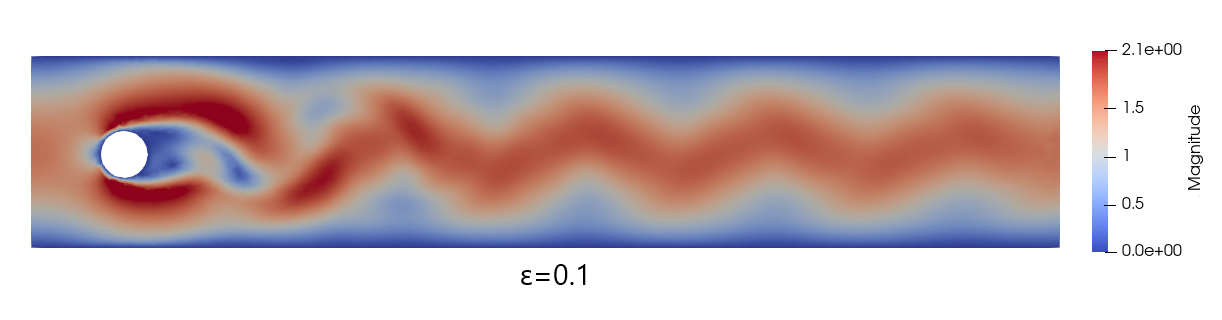} 
  } \\
  \caption{ Ensemble solutions for velocity at time $T=8.8$ for Algorithm \eqref{AlgoCND} with regularization and $\nu = 0.001$, $\gamma = 0.1$ and $\Delta t = 0.001$.\label{CN_flow_Reg_fig} }
\end{figure}

\begin{figure}[H] \captionsetup[subfigure]{labelformat=empty}
  \subfloat[]{%
    \includegraphics[scale=0.45]{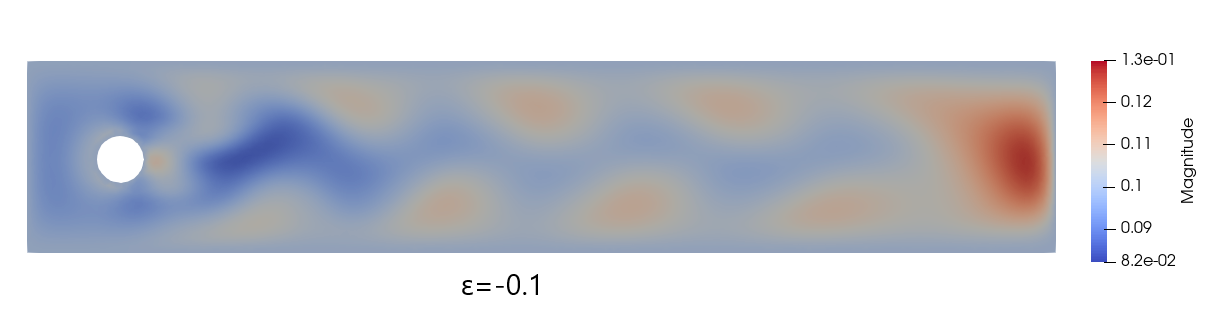} 
  } 
  \qquad\qquad
  \subfloat[]{%
    \includegraphics[scale=0.45]{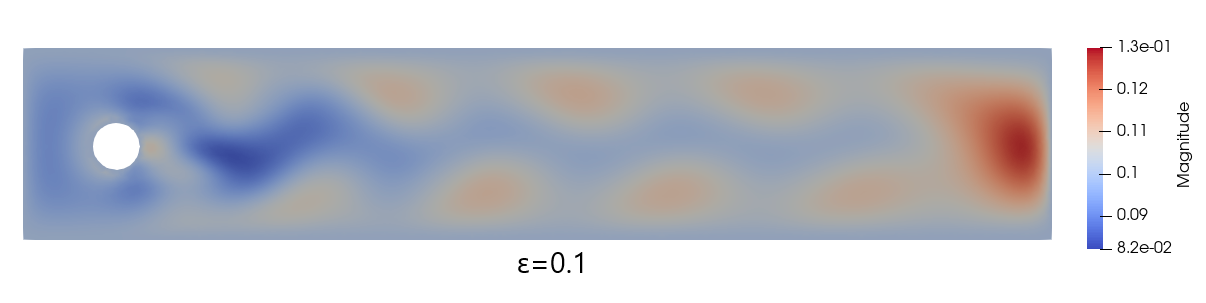} 
  } \\
  \caption{ Ensemble solutions for magnetic field at time $T=8.8$ for Algorithm \eqref{AlgoCND} with regularization and $\nu = 0.001$, $\gamma = 0.1$ and $\Delta t = 0.001$.\label{CN_mag_Reg_fig} }
\end{figure}

\begin{figure}[H] \captionsetup[subfigure]{labelformat=empty}
  \subfloat[]{%
    \includegraphics[scale=0.6]{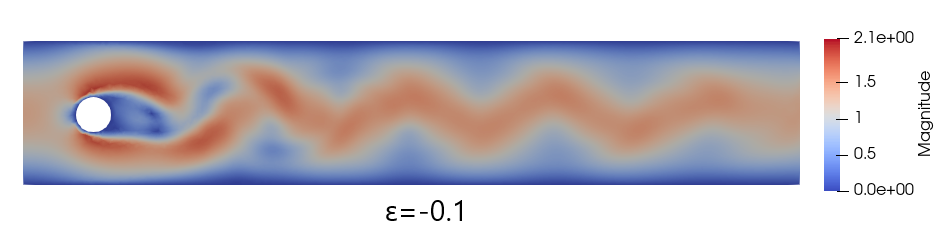} 
  } 
  \qquad\qquad
  \subfloat[]{%
    \includegraphics[scale=0.6]{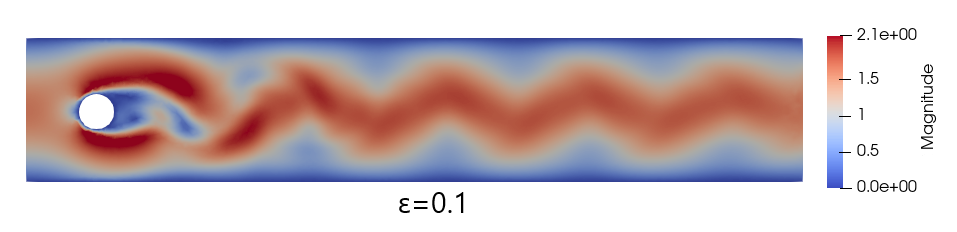} 
  } \\
  \caption{ Ensemble solutions for velocity at time $T=8.8$ for Algorithm \eqref{AlgoBDFD} with regularization and $\nu = 0.001$, $\gamma = 0.1$ and $\Delta t = 0.001$.\label{BDF2_flow_Reg_fig} }
\end{figure}

\begin{figure}[H] \captionsetup[subfigure]{labelformat=empty}
  \subfloat[]{%
    \includegraphics[scale=0.6]{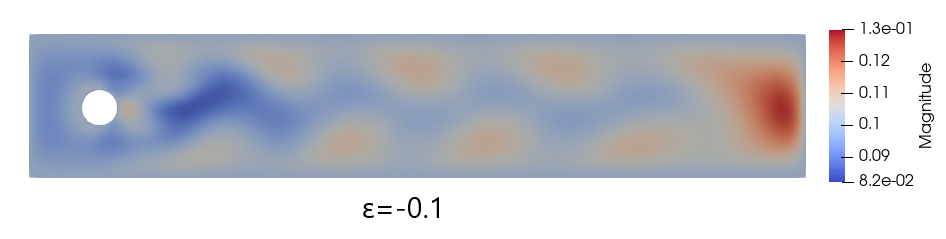} 
  } 
  \qquad\qquad
  \subfloat[]{%
    \includegraphics[scale=0.6]{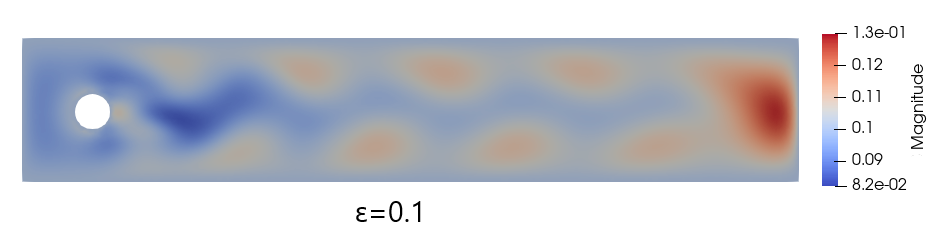} 
  } \\
  \caption{ Ensemble solutions for magnetic field at time $T=8.8$ for Algorithm \eqref{AlgoBDFD} with regularization and $\nu = 0.001$, $\gamma = 0.1$ and $\Delta t = 0.001$.\label{BDF2_mag_Reg_fig} }
\end{figure}

\newpage

\subsection{Accuracy Comparison}
In this section we present a comparison test between the errors of the scheme with and without the regularization terms introduced in Section \ref{ChFlow}. We use the same test as in \ref{convT}, except this time we set $\nu = 1.0$ and $\gamma = 0.2$. We choose two perturbations of $\epsilon = 0.1$ and $\epsilon = 0.2$, with final time $T = 2.5$. For the stabilization coefficients $\alpha$ and $\alpha_M$, we set them equal to the viscosity and magnetic resistivity correspondingly.

\begin{table}[H]
	\centering
	\caption{Error for the first ensemble member in $u_h$.}
	\label{uRegT}
	\begin{tabular}{|c|c|c|c|c|c|}
		\hline
		\rowcolor{lightgray}
		h & $\Delta t$ & SAV-CN & SAV-BDF2 & Stab-SAV-CN & Stab-SAV-BDF2  \\
		\hline
		1/25 & 1/8 & 6.201 e-2 & 3.729 e-2 & 2.865 e-5 & 3.508 e-5\\  
		1/25 & 1/16 & 1.036 e-1 & 6.114 e-2 & 3.087 e-5 & 3.204 e-5 \\
		1/25 & 1/32 & 1.494 e-1 & 9.843 e-2 & 3.223 e-5 & 3.252 e-5 \\
		1/25 & 1/64 & 1.451 e-1 & 1.013 e-1 & 3.261 e-5 & 3.268 e-5 \\
		1/25 & 1/128 & 1.273 e-1 & 9.343 e-2 & 3.271 e-5 & 3.272 e-5 \\
		\hline
		1/100 & 1/8 & 6.306 e-2 & 3.789 e-2 & 6.567 e-6 & 8.040 e-6 \\  
		1/100 & 1/16 & 1.064 e-1 & 6.229 e-2 & 8.492 e-6 & 4.458 e-6 \\
		1/100 & 1/32 & 1.436 e-1 & 8.644 e-2 & 2.178 e-6 & 2.279 e-6 \\
		1/100 & 1/64 & 1.530 e-1 & 9.920 e-2 & 1.098 e-6 & 1.120 e-6 \\
		1/100 & 1/128 & 1.277 e-1 & 8.957 e-2 & 5.982 e-7 & 1.129 e-6 \\
		\hline
	\end{tabular}
\end{table}

\begin{table}[H]
	\centering
	\caption{Error for the first ensemble member in $B_h$.}
	\label{BRegT}
	\begin{tabular}{|c|c|c|c|c|c|}
		\hline
		\rowcolor{lightgray}
		h & $\Delta t$ & SAV-CN & SAV-BDF2 & Stab-SAV-CN & Stab-SAV-BDF2  \\
		\hline
		1/25 & 1/8 & 2.144 e-1 & 1.291 e-1 & 8.263 e-5 & 8.146 e-5 \\  
		1/25 & 1/16 & 3.016 e-1 & 1.926 e-1 & 3.856 e-5 & 3.598 e-5 \\
		1/25 & 1/32 & 3.715 e-1 & 2.419 e-1 & 1.247 e-5 & 1.153 e-5 \\
		1/25 & 1/64 & 3.573 e-1 & 2.433 e-1 & 1.189 e-5 & 1.240 e-5 \\
		1/25 & 1/128 & 3.119 e-1 & 2.174 e-1 & 1.867 e-5 & 1.887 e-5 \\
		\hline
		1/100 & 1/8 & 2.180 e-1 & 1.312 e-1 & 2.372 e-5 & 2.361 e-5\\  
		1/100 & 1/16 & 3.101 e-1 & 1.962 e-1 & 4.695 e-5 & 1.249 e-5 \\
		1/100 & 1/32 & 3.594 e-1 & 2.331 e-1 & 6.428 e-6 & 5.910 e-6 \\
		1/100 & 1/64 & 3.583 e-1 & 2.385 e-1 & 2.726 e-6 & 2.400 e-6 \\
		1/100 & 1/128 & 3.057 e-1 & 2.115 e-1 & 8.643 e-7 & 5.897 e-6 \\
		\hline
	\end{tabular}
\end{table}

\section*{Acknowledgement}
J. Carter was partially supported by the US National Science Foundation grant DMS-1720001 and DMS-1912715.  D. Han was supported by the US National Science Foundation grant  DMS-1912715. N. Jiang was partially supported by the US National Science Foundation grants DMS-1720001 and DMS-2120413. The authors thank  Dr. Suchuan Dong for helpful discussions.

\bibliography{MHD_SAV_bib}
\bibliographystyle{bmc-mathphys.bst}

\end{document}